\documentclass[twoside,reqno]{amsart}

\usepackage{fancyhdr}
\usepackage[latin1]{inputenc}      
\usepackage[T1]{fontenc}
\usepackage{ae,aecompl}
\usepackage{amsmath,amsfonts, amssymb,amsthm} 
\usepackage{mathtools}
\usepackage{units}
\usepackage{latexsym}
\usepackage{dsfont}
\usepackage{enumitem}
\usepackage{hyperref}
\usepackage{cleveref}
\usepackage{microtype}

\DeclareMathOperator*{\wlim}{w-lim}

\numberwithin{equation}{section}

\theoremstyle{plain}
\newtheorem{thm}{Theorem}[section]
\newtheorem{cor}[thm]{Corollary}
\newtheorem{lem}[thm]{Lemma}
\newtheorem{proposition}[thm]{Proposition}

\theoremstyle{definition}
\newtheorem{defn}[thm]{Definition}

\theoremstyle{remark}
\newtheorem{rmk}[thm]{Remark}
\newtheorem{example}[thm]{Example}

\crefname{thm}{Theorem}{Theorems}
\crefname{def}{Definition}{Definitions}
\crefname{rmk}{Remark}{Remarks}
\crefname{lem}{Lemma}{Lemmas}
\Crefname{proposition}{Proposition}{Propositions}
\crefname{cor}{Corollary}{Corollaries}

\renewcommand{\theenumi}{\roman{enumi}}

\renewcommand{\tilde}{\widetilde}
\renewcommand{\emptyset}{\varnothing}
\renewcommand{\rho}{\varrho}

\newcommand{\defeq}{\vcentcolon=}
\newcommand{\eqdef}{=\vcentcolon}
\newcommand{\om}{\omega}
\newcommand{\ee}{\mathrm{e}}
\newcommand{\e}{\varepsilon}
\newcommand{\vp}{\vphantom{\hat{Z}}}
\newcommand{\mdim}{\delta}
\newcommand{\leb}{\lambda}
\newcommand{\bij}{\pi }
\newcommand{\bd}{\rho}
\newcommand{\bdneu}{p}

\newcommand{\card}{\#}
\newcommand{\eigenf}{h}
\newcommand{\eigenv}{\gamma}

\newcommand{\id}{\textup{id}}
\newcommand{\omneu}{u}
\newcommand{\Fun}{F^{\text{unique}}}
\newcommand{\Q}{Q}
\newcommand{\main}{main}
\newcommand{\z}{\zeta}
\newcommand{\aaa}{a}
\newcommand{\bb}{b\,}
\newcommand{\bbb}{b}
\newcommand{\Lomi}{\lvert L_{\om}^i\rvert}
\newcommand{\entro}{H}

\begin{document}

\title[Fractal curvature measures and Minkowski content]{Fractal curvature measures and Minkowski content for one-dimensional self-conformal sets}

\author{Marc Kesseböhmer}\address[Marc Kesseböhmer]{Universität Bremen, Bibliothekstra{\ss}e 1, 28395 Bremen, Germany} \email{mhk@math.uni-bremen.de}    
\author{Sabrina Kombrink}\address[Sabrina Kombrink]{Universität Bremen, Bibliothekstra{\ss}e 1, 28395 Bremen, Germany} \email{kombrink@math.uni-bremen.de}
\date{\today}

\begin{abstract}
  We show that the fractal curvature measures of invariant sets of one-dimensional conformal iterated function systems satisfying the open set condition exist, if and only if the associated geometric potential function is nonlattice. Moreover, in the nonlattice situation we obtain that the Minkowski content exists and prove that the fractal curvature measures are constant multiples of the $\mdim$-conformal measure, where $\mdim$ denotes the Minkowski dimension of the invariant set. For the first fractal curvature measure, this constant factor coincides with the Minkowski content of the invariant set. 
  In the lattice situation we give sufficient conditions for the Minkowski content of the invariant set to exist, contrasting the fact that the Minkowski content of a self-similar lattice fractal never exists. However, every self-similar set satisfying the open set condition exhibits a Minkowski measurable $\mathcal{C}^{1+\alpha}$ diffeomorphic image.
	Both in the lattice and nonlattice situation average versions of the fractal curvature measures are shown to always exist.
  
\end{abstract}
\maketitle

\section{Brief Introduction}
	Notions of curvature are an important tool to describe the geometric structure of sets and have been introduced and intensively studied for broad classes of sets. 
Originally, the idea to characterise sets in terms of their curvature stems from the study of smooth manifolds as well as from the theory of convex bodies with sufficiently smooth boundaries. In his fundamental paper \textit{Curvature Measures} \cite{Federercurvature}, Federer localises, extends and unifies the previously existing notions of curvature to sets of positive reach. This is where he introduces curvature measures, which can be viewed as a measure theoretical substitute for the notion of curvature for sets without a differentiable structure.
Federer's curvature measures were studied and generalised in various ways. An extension to finite unions of convex bodies is given in \cite{Groemer} and \cite{Schneider_polyconvex} and to finite unions of sets with positive reach in \cite{Zaehle_posreach}. 
 In  \cite{Winter_thesis}, Winter extends the curvature measures to fractal sets in $\mathbb{R}^{d}$, which typically cannot be expressed as finite unions of sets with positive reach. These measures are referred to as fractal curvature measures and are defined as weak limits of rescaled versions of the curvature measures introduced by Federer, Groemer and Schneider.  Winter also examines conditions for their existence in the self-similar case.
 However, fractal sets arising in geometry (for instance as limit sets of Fuchsian groups) or in number theory (for instance as sets defined by Diophantine inequalities) are typically non self-similar but rather self-conformal. 
 In order to have a notion of curvature at hand also for this important class of fractal sets, in this paper, we study nonempty compact sets which occur as the invariant sets of finite conformal iterated function systems satisfying the open set condition and call them self-conformal sets (see \cref{def:selfconformal}). We obtain necessary and sufficient conditions for the fractal curvature measures to exist for these kind of sets and establish links to the Minkowski content.

The Minkowski content was proposed in \cite{Mandelbrot_lacunarity} as a measure of ``lacunarity'' for a fractal set. Indeed, the value of the Minkowski content allows to compare the lacunarity of sets of the same Minkowski dimension.
Minkowski measurability has moreover attracted prominence in work related to the Weyl-Berry conjecture on the distribution of eigenvalues of the Laplacian on domains with fractal boundaries. We refer to Section 4 of \cite{Falconer_Minkowski} for an overview and references concerning these studies.
An additional motivation for studying the Minkowski content of fractal sets arises from noncommutative geometry. 
In Connes' seminal book \cite{Connes_seminal} the notion of a noncommutative fractal geometry is developed. There it is shown that the natural analogue of the volume of a compact smooth Riemannian spin manifold for a fractal set in $\mathbb R$ is that of the Minkowski content. This idea is also reflected in the works \cite{Guido_Isola}, \cite{Samuel} and \cite{FalconerSamuel}.  

The paper is organised as follows. 
In Section \ref{sec:results} we state the main results and provide in this way a complete answer to the question on the existence of the fractal curvature measures for self-conformal sets. The precise definitions and background information as well as the relevant properties and auxiliary results will be presented in Section \ref{sec:setting}. In Section \ref{sec:proofs} the proofs of our main theorems for self-conformal sets (\cref{curvatureresult,conformalMinkowski}) are provided. Finally, in Section \ref{sec:specialcases}, we conclude the paper by considering the special cases of self-similar sets and $\mathcal C^{1+\alpha}$ diffeomorphic images of self-similar sets and prove \cref{similars,corimage,corminim}.

\section{Main Results}\label{sec:results}

The introduction of the fractal curvature measures (see Section \ref{sec:fcm}) relies on the definition of scaling exponents, for which we require the following notation. Let $\leb^0$ and $\leb^1$ respectively denote the zero- and one-dimensional Lebesgue measure. For $\e>0$ we define $Y_{\e}\defeq\{x\in\mathbb R\mid\inf_{y\in Y}\lvert x-y\rvert\leq\e\}$ to be the \emph{$\e$-parallel neighbourhood} of $Y\subset\mathbb R$ and let $\partial Y$ denote the boundary of $Y$.
\begin{defn}\label{scalingexp}
	For a nonempty compact set $Y\subset\mathbb R$ the \emph{0-th} and \emph{1-st curvature scaling exponents} of $Y$ are respectively defined as 
	\begin{eqnarray*}
		s_0(Y)&\defeq& \inf\{t\in\mathbb R\mid \e^t\leb^0(\partial Y_{\e})\to 0\ \text{as}\ \e\to 0\}\quad\text{and}\\
		s_1(Y)&\defeq& \inf\{t\in\mathbb R\mid \e^t\leb^1(Y_{\e})\to 0\ \text{as}\ \e\to 0\}.
	\end{eqnarray*}
\end{defn}

\begin{defn}\label{defn:curvaturemeasures}
	Let $Y\subset\mathbb R$ denote a nonempty compact set. Provided, the weak limit of finite Borel measures
	\[
		C^f_0(Y,\cdot)\defeq\wlim_{\e\to 0} \e^{s_0(Y)}\leb^0(\partial Y_{\e}\cap\cdot)/2
	\]
	exists, we call it the \emph{0-th fractal curvature measure} of $Y$. Likewise the weak limit
	\[
		C_1^f(Y,\cdot)\defeq\wlim_{\e\to 0}\e^{s_1(Y)}\leb^1(Y_{\e}\cap\cdot)
	\]
	is called the \emph{1-st fractal curvature measure}, if it exists. Moreover, for a Borel set $B\subseteq\mathbb R$ we set 
	\[
	\begin{array}{ll}
		\overline{C}_0^f(Y,B)\defeq\displaystyle{\limsup_{\e\to 0}\e^{s_0(Y)}\leb^0(\partial Y_{\e}\cap B)/2},&
		\overline{C}_1^f(Y,B)\defeq\displaystyle{\limsup_{\e\to 0}\e^{s_1(Y)}\leb^1(Y_{\e}\cap B)},\\
		\underline{C}_0^f(Y,B)\defeq\displaystyle{\liminf_{\e\to 0}\e^{s_0(Y)}\leb^0(\partial Y_{\e}\cap B)/2},&
		\underline{C}_1^f(Y,B)\defeq\displaystyle{\liminf_{\e\to 0}\e^{s_1(Y)}\leb^1(Y_{\e}\cap B)}.
	\end{array}
	\]
\end{defn}

The central question arising in this context is to identify those sets $Y\subset \mathbb{R}$ for which the fractal curvature measures exist. In \cite{Winter_thesis} it has been shown that the fractal curvature measures exist for self-similar sets with positive Lebesgue measure as well as for self-similar sets which are nonlattice (see \cref{def:nonlatticefractal}) and satisfy the open set condition (see Section \ref{sec:selfconfshift}).
In the lattice case, Winter shows that average versions of the fractal curvature measures exist, which  are defined as follows.
\begin{defn}\label{def:averagefc}
	Let $Y\subset\mathbb R$ denote a nonempty compact set. Provided the weak limit of finite Borel measures exists, we let
	\[
		\widetilde{C}_0^f(Y,\cdot)\defeq\wlim_{T\searrow 0}\lvert\ln T\rvert^{-1}\int_T^1\e^{s_0(Y)-1}\leb^0(\partial Y_{\e}\cap\cdot)\textup{d}\e/2
	\] 
	denote the \emph{0-th average fractal curvature measure} of $Y$ and let the weak limit 
	\[
		\widetilde{C}_1^f(Y,\cdot)\defeq\wlim_{T\searrow 0} \lvert\ln T\rvert^{-1}\int_T^1\e^{s_1(Y)-1}\leb^1(Y_{\e}\cap\cdot)\textup{d}\e	
	\]
	likewise denote the \emph{1-st average fractal curvature measure} of $Y$.
\end{defn}

Note that the definition of the 1-st curvature scaling exponent resembles the definition of the Minkowski dimension, which is also known as the box counting dimension (see Claim 3.1 in \cite{Falconer_Foundation}), and defined as follows.

\begin{defn}
	For a nonempty compact set $Y\subset\mathbb R$ the \emph{upper} and \emph{lower Minkowski dimension} is respectively defined as 
	\begin{eqnarray*}
		\overline{\textup{dim}}_{M}(Y)&\defeq& 1-\liminf_{\e\searrow 0}\frac{\ln\leb^1(Y_{\e})}{\ln\e}\qquad\text{and}\\
		\underline{\textup{dim}}_{M}(Y)&\defeq& 1-\limsup_{\e\searrow 0}\frac{\ln\leb^1(Y_{\e})}{\ln\e}.
	\end{eqnarray*}
	In case the upper and lower Minkowski dimensions coincide, we call the common value the \emph{Minkowski dimension} of $Y$ and denote it by $\textup{dim}_M(Y)\eqdef\mdim$.
\end{defn}

In what follows, we provide a complete characterisation of self-conformal sets for which the (average) fractal curvature measures exist.
By a self-conformal set we mean a set which arises as the invariant set of a finite iterated function system which consists of $\mathcal C^{1+\alpha}$ maps (see \cref{def:selfconformal}). For such sets it is well-known that the Minkowski dimension exists (see \cref{thBedford}).
As we will see, a self-conformal set is either a nonempty compact interval or has zero one-dimensional Lebesgue measure (\Cref{intervalfractal}). In order to determine the fractal curvature scaling exponents we have to distinguish these two cases. 

\begin{proposition}\label{prop:interval}
	Let $\mdim$ denote the Minkowski dimension of a self-conformal set $F$. If $\leb^1(F)=0$, then $s_0(F)=\mdim$ and $s_1(F)=\mdim-1$. If $F$ is a nonempty compact interval, then $s_0(F)=s_1(F)=0$.
\end{proposition}
Let us first consider the latter situation of the above proposition. 
As an immediate consequence of \Cref{prop:interval} we obtain the following complete description.
\begin{cor}\label{prop:curvaturetrivial}
	If $Y\subset\mathbb R$ is a nonempty compact interval, then both the 0-th and 1-st fractal curvature measures exist and satisfy
	\begin{eqnarray*}
		C^f_0(Y,\cdot)=\leb^0(\partial Y\cap\cdot)/2\quad \text{and}
		\quad C^f_1(Y,\cdot)=\leb^1(Y\cap\cdot).
	\end{eqnarray*}
\end{cor}

Let us now focus on self-conformal sets with zero one-dimensional Lebesgue measure. 
Here, as the total mass of the (average) fractal curvature measure, the (average) Minkowski content appears. This is defined as follows.

\begin{defn}\label{defn:Mcontent}
	Let $Y\subset\mathbb R$ denote a set for which the Minkowski dimension $\mdim$ exists. 
	The \emph{upper Minkowski content} $\overline{\mathcal M}(Y)$ and the \emph{lower Minkowski content} $\underline{\mathcal M}(Y)$ of $Y$ are respectively defined as 
	\begin{eqnarray*}
		\overline{\mathcal M}(Y)
		\defeq\limsup_{\e\to 0}\e^{\mdim-1}\leb^1(Y_{\e})\quad\text{and}\quad
		\underline{\mathcal M}(Y)
		\defeq\liminf_{\e\to 0}\e^{\mdim-1}\leb^1(Y_{\e}).
	\end{eqnarray*}
	If the upper and lower Minkowski contents coincide, we denote the common value by $\mathcal M(Y)$ and call it the \emph{Minkowski content} of $Y$. In case the Minkowski content exists, we call $Y$ \emph{Minkowski measurable}. The \emph{average Minkowski content} of $Y$ is defined as the following limit if it exists
	\[
		\widetilde{\mathcal M}(Y)\defeq\lim_{T\searrow 0}\lvert\ln T\rvert^{-1}\int_T^1\e^{\mdim-2}\leb^1(Y_{\e})\textup{d}\e.
	\]
\end{defn}

For the following fix an iterated function system $\Phi\defeq\{\phi_1,\ldots,\phi_N\}$,  $N\geq2$, acting on a nonempty compact connected set $X\subset\mathbb R$ which satisfies the open set condition (see Section \ref{sec:selfconfshift}) with feasible open set $\textup{int} X$, where $\textup{int}X$ shall denote the interior of $X$. Let $F$ denote the unique nonempty compact invariant set of $\Phi$ and call it self-conformal set associated with $\Phi$. For  $\Sigma\defeq\{1,\ldots,N\}$  let $(\Sigma^{\infty},\sigma)$ denote the full shift-space on $N$ symbols and let $\bij\colon\Sigma^{\infty}\to F$ be the natural code map as defined in Section \ref{sec:selfconfshift}.
It turns out that the fractal curvature measures of $F$ exist, if and only if the \emph{geometric potential function} $\xi\colon\Sigma^{\infty}\to\mathbb R$ given by $\xi(\om)\defeq -\ln\lvert\phi'_{\om_1}(\bij(\sigma\om))\rvert$ for $\om\defeq\om_1\om_2\cdots\in\Sigma^{\infty}$, is nonlattice (see \cref{def:nonlattice}). In this case we call $\Phi$ (resp. $F$) nonlattice, otherwise $\Phi$ (resp. $F$) is called lattice (see \cref{def:nonlatticefractal}).

By applying $\Phi$ to the convex hull of $F$ one obtains a family of $Q-1$ gap intervals  $L^1,\ldots,L^{Q-1}$, which we call the \emph{primary gaps} of $F$, where we have $2\leq Q\leq N$ since $\leb^1(F)=0$. Given an $n\in\mathbb N$ and an $\om\defeq\om_1\cdots\om_n\in\Sigma^n$, let $L^1_{\om},\ldots,L^{Q-1}_{\om}$ respectively denote the images of the primary gaps under the map $\phi_{\om}\defeq\phi_{\om_1}\circ\cdots\circ\phi_{\om_n}$ and call these sets the \emph{\main\ gaps} of $\phi_{\om}F$.\\ 
Further, letting $\mdim$ denote the Minkowski dimension of $F$, we call the unique probability measure $\nu$ supported on $F$, which satisfies
\begin{eqnarray}\label{conformalmeasure}
	\nu(\phi_i X\cap\phi_j X)=0\ \text{for}\ i\neq j\in\Sigma\quad\text{and}\quad	\nu(\phi_{i}B)=\int_B \lvert\phi'_{i}\rvert^{\mdim}\textup{d}\nu
\end{eqnarray}
for all $i\in\{1,\ldots,N\}$ and for all Borel sets $B\subseteq X$ the \emph{$\mdim$-conformal measure} associated with $\Phi$. The statement on the uniqueness and existence is shown in \cite{MauldinUrbanski} and goes back to the work of \cite{Patterson}, \cite{Sullivan}, and \cite{DenkerUrbanski}.\\
Finally, let $\entro_{\mu_{-\mdim\xi}}$ denote the measure theoretical entropy of the shift-map with respect to the unique shift-invariant Gibbs measure for the potential function $-\mdim\xi$ (see (\ref{entropy}) in Section \ref{sec:PF}).

The complete answer to the question concerning the existence of the (average) fractal curvature measures for self-conformal sets is given in the following theorem which shows that the fractal curvature measures exist if and only if $\xi$ is nonlattice.

\begin{thm}[Self-Conformal Sets -- Fractal Curvature Measures]\label{curvatureresult}
	Let $F$ denote a self-conformal set associated with the iterated function system $\Phi$ acting on $X$. Assume that $\Phi$ satisfies the open set condition with feasible open set $\textup{int} X$ and that $\leb^1(F)=0$.
	Let $\mdim$ denote the Minkowski dimension of $F$ and let $\xi$ denote the geometric potential function associated with $\Phi$. Then the following hold.
	\begin{enumerate}
	\item\label{curvatureresult:average} The average fractal curvature measures always exist and are both constant multiples of the $\mdim$-conformal measure $\nu$ associated with $F$, that is
		\begin{eqnarray*}
		\hspace{2cm}\widetilde{C}_0^f(F,\cdot)=\frac{2^{-\mdim}c}{\entro_{\mu_{-\mdim\xi}}}\cdot\nu(\cdot)\quad\text{and}\quad
		\widetilde{C}_1^f(F,\cdot)=\frac{2^{1-\mdim}c}{(1-\mdim)\entro_{\mu_{-\mdim\xi}}}\cdot\nu(\cdot), 
		\end{eqnarray*}
		where the constant $c>0$ is given by the well-defined limit
	 	\begin{eqnarray}\label{constantthm}
	 		c\defeq \lim_{n\to\infty}\sum_{i=1}^{Q-1}\sum_{\om\in\Sigma^n}\Lomi^{\mdim}.
	 	\end{eqnarray}
	 \item\label{curvatureresult:nonlattice} If $\Phi$ is nonlattice, then both the 0-th and 1-st fractal curvature measures exist and satisfy $C_k^f(F,\cdot)=\widetilde{C}_k^f(F,\cdot)$ for $k\in\{0,1\}$.
	\item\label{curvatureresult:lattice} If $\Phi$ is lattice, then neither the 0-th nor the 1-st fractal curvature measure exists. Nevertheless, there exists a constant $\overline{c}\in\mathbb R$ such that $\overline{C}^f_k(F,B)\leq \overline{c}$ for every Borel set $B\subseteq\mathbb R$ and $k\in\{0,1\}$. Additionally, $\underline{C}^f_k(F,\mathbb R)$ is positive for $k\in\{0,1\}$.
	\end{enumerate}
\end{thm}
Note that Parts (\ref{curvatureresult:nonlattice}) and (\ref{curvatureresult:lattice}) in particular show that the scaling exponents of $F$ can alternatively be characterised by $s_0(F)= \sup\{t\in\mathbb R\mid \e^t\leb^0(\partial F_{\e})\to\infty\ \text{as}\ \e\to 0\}$ and $s_1(F)=\sup\{t\in\mathbb R\mid \e^t\leb^1(F_{\e})\to\infty\ \text{as}\ \e\to 0\}$ respectively. 

Using the definition of the Minkowski content and \Cref{prop:interval}, we see that the existence of the fractal curvature measures immediately implies the existence of the Minkowski content. But it is important to remark that the fact that the fractal curvature measures do not exist in the lattice case does not imply that the Minkowski content does not exist in this case. Indeed, for a general self-conformal set, which is lattice, the Minkowski content may or may not exist.
A sufficient condition under which the Minkowski content exists is given in Part (\ref{conformalMinkowski:lattice}) of the following theorem. Parts (\ref{conformalMinkowski:average}) and (\ref{conformalMinkowski:nonlattice}) of the following theorem are immediate consequences of \cref{curvatureresult}.

 Let $\Sigma^{\infty}$ be equipped with the topology of pointwise convergence and let $\mathcal C (\Sigma^{\infty})$ denote the space of continuous real valued functions on $\Sigma^{\infty}$. For an $\alpha$-Hölder continuous function $f\in\mathcal F_{\alpha}(\Sigma^{\infty})$ (see Section \ref{sec:PF}) we let $\nu_f$ denote the unique eigenmeasure corresponding to the eigenvalue 1 of the dual of the Perron-Frobenius operator for the potential function $f$ (see Section \ref{sec:PF}).

\begin{thm}[Self-Conformal Sets -- Minkowski Content]\label{conformalMinkowski}
  Under the conditions of \cref{curvatureresult} and letting $c$ denote the constant given in Equation (\ref{constantthm}), the following hold.	
  \begin{enumerate}
  \item\label{conformalMinkowski:average} The average Minkowski content exists and equals 
    \begin{eqnarray*}
      \widetilde{\mathcal{M}}(F)=\frac{2^{1-\mdim}c}{(1-\mdim)\entro_{\mu_{-\mdim\xi}}}.
    \end{eqnarray*}
  \item\label{conformalMinkowski:nonlattice} If $\Phi$ is nonlattice, then the Minkowski content $\mathcal M(F)$ of $F$ exists and coincides with $\widetilde{\mathcal{M}}(F)$.
  \item\label{conformalMinkowski:lattice} If $\Phi$ is lattice, then we have that
    \[
    0<\underline{\mathcal M}(F)\leq\overline{\mathcal M}(F)<\infty.
    \]
    What is more, equality in the above equation can be attained. More precisely let $\z,\psi\in\mathcal{C}(\Sigma^{\infty})$ denote the functions satisfying $\xi-\z=\psi-\psi\circ\sigma$, where the range of $\z$ is contained in a discrete subgroup of $\mathbb R$ and $\aaa\in\mathbb R$ is maximal such that $\z(\Sigma^{\infty})\subseteq\aaa\mathbb Z$. If, for every $t\in[0,\aaa)$, we have 
      \begin{eqnarray}\label{existencecondition}
	&&\sum_{n\in\mathbb Z}\ee^{-\mdim\aaa n}\nu_{-\mdim\z}\circ\psi^{-1}([n\aaa,n\aaa+t))\nonumber\\
	&&\qquad\qquad=\frac{\ee^{\mdim t}-1}{\ee^{\mdim\aaa}-1}\sum_{n\in\mathbb Z}\ee^{-\mdim\aaa n}\nu_{-\mdim\z}\circ\psi^{-1}([n\aaa,(n+1)\aaa)),
      \end{eqnarray}
      then it follows that $\underline{\mathcal{M}}(F)=\overline{\mathcal{M}}(F)$.
  \end{enumerate}
\end{thm}
Note that the sums occuring in Equation (\ref{existencecondition}) are finite.

\begin{rmk}
  Condition (\ref{existencecondition}) in fact not only implies the existence of the Minkowski content but also that $\underline{C}_0^f(F,\mathbb R)=\overline{C}_0^f(F,\mathbb R)$ (see proof of Part (\ref{conformalMinkowski:lattice}) of \cref{conformalMinkowski}).
\end{rmk}
An example of a self-conformal set $F$, which satisfies Condition (\ref{existencecondition}) and thus is Minkowski measurable, is given in \Cref{thexample}.
However, in the special case when $F$ is a self-similar set, Condition (\ref{existencecondition}) cannot be satisfied. In this case it even turns out, that $F$ is Minkowski measurable if and only if $F$ is nonlattice.
This is also reflected in the following theorem.  

\begin{thm}[Self-Similar Sets -- Fractal Curvature Measures]\label{similars}
	Let $F$ denote a self-similar set associated with the iterated function system $\Phi\defeq\{\phi_1,\ldots,\phi_N\}$ acting on $X$. Assume that $\Phi$ satisfies the open set condition with feasible open set $\textup{int} X$ and that $\leb^1(F)=0$. Further, let $r_1,\ldots,r_N$ denote the respective similarity ratios of the maps $\phi_1,\ldots,\phi_N$.
	Let $\mdim$ denote the Minkowski dimension of $F$ and let $\nu$ be the $\mdim$-conformal measure associated with $F$. Then, additionally to the statements of \cref{curvatureresult}, the following hold. 
	\begin{enumerate}
	\item\label{ss:average} 
	The formulae for the average fractal curvature measures simplify to
\begin{eqnarray*}
	\widetilde{C}_0^f(F,\cdot) &=& \frac{2^{-\mdim}\sum_{i=1}^{\Q-1}\lvert L^{i}\rvert^{\mdim}}{-\mdim\sum_{i\in\Sigma}\ln(r_i)r_i^{\mdim}}\cdot \nu(\cdot)\qquad\quad \text{and}\quad\\
	\widetilde{C}_1^f(F,\cdot) &=& 	\frac{2^{1-\mdim}\sum_{i=1}^{\Q-1}\lvert L^{i}\rvert^{\mdim}}{(\mdim-1)\mdim\sum_{i\in\Sigma}\ln(r_i)r_i^{\mdim}}\cdot \nu(\cdot).
\end{eqnarray*}
	\item\label{ss:lattice} If $\Phi$ is lattice, the following holds, strengthening Part (\ref{curvatureresult:lattice}) of \cref{curvatureresult}. For $k\in\{0,1\}$ and for every Borel set $B\subseteq\mathbb R$ for which $F\cap B$ is nonempty and allows a representation as a finite union of sets of the form $\phi_{\om}F$, where $\om\in\Sigma^*$, and for which $F_{\e}\cap B=(F\cap B)_{\e}$ for all sufficiently small $\e>0$ we have
\begin{eqnarray*}
	0<\underline{C}_k^f(F,B)<\overline{C}_k^f(F,B)<\infty.
\end{eqnarray*}
	\end{enumerate}
\end{thm}

 Note that $\nu$ coincides with the $\mdim$-dimensional Hausdorff measure normalised on $F$, that is with $\mathcal{H}^{\mdim}(\cdot\cap F)/\mathcal{H}^{\mdim}(F)$.

For self-similar sets the existence of the average fractal curvature measures  and of the fractal curvature measures in the nonlattice case has also been shown in Theorem 1.2.6 of \cite{Winter_thesis}. However, the formulae for the coefficients of the measures obtained in \cite{Winter_thesis} are given by an integration over a certain ``overlap function'' and appear to be much harder to determine explicitely. The statement concerning the lattice case is not covered by \cite{Winter_thesis} and gives a new result.
For the Minkowski content \cref{similars} immediately implies the following corollary which we state without a proof.

\begin{cor}[Self-Similar Sets -- Minkowski Content]\label{similarmink}
	Under the conditions of \cref{similars} the following hold.
\begin{enumerate}
	\item The average Minkowski content of $F$ exists and is given by
\[
	\widetilde{\mathcal M}(F)=\frac{2^{1-\mdim}\sum_{i=1}^{\Q-1}\lvert L^{i}\rvert^{\mdim}}{(\mdim-1)\mdim\sum_{i\in\Sigma}\ln(r_i)r_i^{\mdim}}.
\]	
	\item\label{ssMinkowski:nonlattice} If $\Phi$ is nonlattice, its Minkowski content $\mathcal M(F)$ exists and is equal to $\widetilde{\mathcal M}(F)$.
	\item\label{ssMinkowski:lattice} If $\Phi$ is lattice, then 
	\[
		0<\underline{\mathcal{M}}(F)<\overline{\mathcal{M}}(F)<\infty.
	\]
\end{enumerate}
\end{cor}

Part (\ref{ssMinkowski:nonlattice}) of \cref{similarmink} has been obtained in Proposition 4 of \cite{Falconer_Minkowski} 
under the strong seperation condition. Part (\ref{ssMinkowski:lattice}) of \cref{similarmink} has also been addressed in Theorem 8.36 of \cite{Lapidus_Frankenhuysen_Springer}.

Another special case of self-conformal sets are $\mathcal C^{1+\alpha}$ diffeomorphic images of self-similar sets. Here, $\mathcal C^{1+\alpha}$ is the class of real valued functions which are differentiable with Hölder continuous derivative.
For these sets Parts (\ref{curvatureresult:average}) and (\ref{curvatureresult:nonlattice}) of \cref{curvatureresult} yield interesting relationships between the (average) fractal curvature measures of the self-similar set and of its $\mathcal C^{1+\alpha}$ diffeomorphic image which are stated in the following corollary. 

\begin{thm}[$\mathcal C^{1+\alpha}$ Images -- Fractal Curvature Measures]\label{corimage}
  Let $K\subset\mathbb R$ denote a self-similar set for the iterated function system $\Phi$ acting on $X$ which satisfies the open set condition with feasible open set $\textup{int} X$. Let $\mdim$ denote its Min\-kowski dimension, let $\mathcal U\supset X$ be a connected open neighbourhood of $X$ in $\mathbb R$ and $g\colon \mathcal U\to\mathbb R$ be a $\mathcal{C}^{1+\alpha}(\mathcal U)$ map, $\alpha>0$, for which $|g'|$ is bounded away from 0. Assume that $\leb^1(K)=0$ and set $F\defeq g(K)$.
  \begin{enumerate}
  \item\label{image:average} The average fractal curvature measures of both $K$ and $F$ exist. Moreover, they are absolutely continuous and for $k\in\{0,1\}$ their Radon-Nikodym derivatives are given by 
    \begin{eqnarray*}
      \frac{\textup{d}\widetilde{C}_k^f(F,\cdot)}{\textup{d}\widetilde{C}_k^f(K,\cdot)\circ g^{-1}}=\lvert g'\circ g^{-1}\rvert^{\mdim}.
    \end{eqnarray*}
  \item\label{image:nonlattice} If $\Phi$ is nonlattice, then the fractal curvature measures of both $K$ and $F$ exist and are absolutely continuous with Radon-Nikodym derivatives
    \begin{eqnarray*}
      \frac{\textup{d}C_k^f(F,\cdot)}{\textup{d}C_k^f(K,\cdot)\circ g^{-1}}=\lvert g'\circ g^{-1}\rvert^{\mdim},
    \end{eqnarray*}
    for $k\in\{0,1\}$.
  \end{enumerate}
\end{thm}

In contrast to the self-similar setting, the Minkowski content of a $\mathcal C^{1+\alpha}$ diffeomorphic image of a lattice self-similar set may or may not exist. In fact, for every lattice fractal self-similar set $K$ there exist diffeomorphisms $g\in\mathcal{C}^{1+\alpha}$ such that $g(K)$ is Minkowski measurable. The explicit form of such diffeomorphisms is given in Part (\ref{minim:lattice}) of the next corollary. Parts (\ref{minim:average}) and (\ref{minim:nonlattice}) of the next corollary are immediate consequences of \cref{corimage}.

\begin{cor}[$\mathcal C^{1+\alpha}$ Images -- Minkowski Content]\label{corminim}
  Assume the conditions of \cref{corimage} and let $\nu$ denote the $\mdim$-conformal measure associated with $K$. Then we have the following.
  \begin{enumerate}
  \item\label{minim:average} The average Minkowski contents of both $K$ and $F$ exist and satisfy
    \[
    \widetilde{\mathcal M}(F)=\widetilde{\mathcal M}(K)\cdot \int \lvert g'\rvert^{\mdim}\textup{d}\nu.
    \]
  \item\label{minim:nonlattice} If $\Phi$ is nonlattice, then the Minkowski content of both $K$ and $F$ exist and satisfy 
    \[
    \mathcal M(F)=\mathcal M(K)\cdot \int \lvert g'\rvert^{\mdim}\textup{d}\nu.
    \]
  \item\label{minim:lattice} If $\Phi$ is lattice, then the Minkowski content of $K$ does not exist, whereas the Minkowski content of $F$ might or might not exist. More precisely, assume that $K\subseteq[0,1]$ and that the geometric potential function $\xi$ associated with $\Phi$ is lattice. Let $\aaa>0$ be maximal such that the range of $\xi$ is contained in $\aaa\mathbb Z$. Define $\tilde{g}\colon\mathbb R\to\mathbb R$, $\tilde{g}(x)\defeq\nu((-\infty,x])$ to be the distribution function of $\nu$. For $n\in\mathbb N$ define the function $g_n\colon[-1,\infty)\to\mathbb R$ by
\[
g_n(x)\defeq\int_{-1}^x\left(\tilde{g}(r)(\ee^{\mdim\aaa n}-1)+1\right)^{-1/\mdim}\textup{d}r
\]
and set $F_n\defeq g_n(K)$. Then for every $n\in\mathbb N$ we have $\underline{\mathcal{M}}(F_n)=\overline{\mathcal{M}}(F_n)$.
  \end{enumerate}
\end{cor}
\begin{rmk}
  The sets $F_n$ constructed in Part (\ref{minim:lattice}) of \cref{corminim} are actually not only Minkowski measurable but also satisfy $\underline{C}_0^f(F_n)=\overline{C}_0^f(F_n)$ (see proof of Part (\ref{minim:lattice}) of \cref{corminim}).
\end{rmk}

The results stated in \cref{corimage,corminim} have recently been obtained also for higher dimensions in \cite{Uta}. There, $\mathcal C^{1+\alpha}$ diffeomorphic images of self-similar sets satisfying the strong seperation condition are considered.

The above results enable us to construct examples of lattice self-conformal sets for which the Minkowski content exist.

\begin{example}\label{thexample}
  Let $K\subseteq[0,1]$ be the Middle Third Cantor Set and let $\nu$ denote the $\ln2/\ln 3$-conformal measure associated with $K$. 
Let $\tilde{g}\colon\mathbb R\to\mathbb R$ denote the Devil's Staircase Function defined by $\tilde{g}(r)\defeq\nu((-\infty,r])$, define the function $g\colon[-1,\infty)\to\mathbb R$ by 
\[
g(x)\defeq\int_{-1}^{x}(\tilde{g}(y)+1)^{-\ln3/\ln2}\textup{d}y
\]
and set $F\defeq g(K)$. Then we have $\underline{\mathcal M}(F)=\overline{\mathcal M}(F)$, although $\underline{\mathcal M}(K)<\overline{\mathcal M}(K)$. This is a consequence of \cref{similarmink,corminim}.
\end{example}

\section{Preliminaries}\label{sec:setting}

\subsection{Fractal Curvature Measures}\label{sec:fcm}

The work of \cite{Groemer} and \cite{Schneider_polyconvex} plays a vital role in the introduction of Winter's fractal curvature measures. In what follows, we focus on the construction in the one-dimensional setting.
For a set $Y\subset\mathbb R$ which is a finite union of compact convex sets, there exist two curvature measures, namely the 0-th and the 1-st curvature measure of $Y$. Originally, these measures were defined through a localised Steiner formula (see \cite{Federercurvature} and \cite{Schneider_polyconvex}), but an equivalent and simpler characterisation is the following. The 1-st curvature measure of $Y$ equals $\leb^1(Y\cap\cdot)$ and under the additional assumption that $Y$ is the closure of its interior, the 0-th curvature measure is equal to $\leb^0(\partial Y\cap\cdot)/2$. 

If $Y\subset\mathbb R$ is not a finite union of compact convex sets, but an arbitrary compact set, we still have that the $\e$-parallel neighbourhood $Y_{\e}$ of $Y$ is a finite union of convex compact sets, for each $\e>0$. Moreover, $Y_{\e}$ is the closure of its interior, for each $\e>0$. Thus, the 0-th and 1-st curvature measures are defined on $Y_{\e}$ and are equal to the measures $\leb^0(\partial Y_{\e}\cap\cdot)/2$ and $\leb^1(Y_{\e}\cap\cdot)$. 
The fractal curvature measures now arise by taking the weak limit as $\e\to 0$. However, before taking the limit, we observe that for a fractal set $F\subset\mathbb R$ one typically obtains that the number of boundary points of $F_{\e}$ tends to infinity as $\e\to 0$, whereas the volume of $F_{\e}$ tends to zero as $\e\to 0$.
In order to obtain nontrivial measures, we need to introduce the curvature scaling exponents $s_0(F)$ and $s_1(F)$ as in \cref{scalingexp}.
By taking the weak limits of the rescaled curvature measures $\e^{s_0(F)}\cdot\leb^0(\partial F_{\e}\cap\,\cdot\,)/2$ and $\e^{s_1(F)}\cdot \leb^1(F_{\e}\cap\,\cdot\,)$ as $\e\to 0$, we obtain the fractal curvature measures $C_0^f(F,\cdot)$ and $C_1^f(F,\cdot)$ (\cref{defn:curvaturemeasures}), whenever the weak limits exist. The average fractal curvature measures are gained by taking the weak limit over the average rescaled curvature measures if these limits exist (\cref{def:averagefc}).

Besides extending the notions of curvature, the fractal curvature measures also provide a set of geometric characteristics of a fractal set which can be used to distinguish fractal sets of the same Minkowski dimension. 
More precisely, considering two fractal sets $F_1,F_2\subseteq [0,1]$ with $\{0,1\}\subseteq F_1,F_2$ which are of the same Minkowski dimension, the 1-st fractal curvature measure compares the local rate of decay of the lengths of the $\e$-parallel neighbourhood of $F_1$ and $F_2$. In this way it can be interpreted as ``local fractal length". Since, by the inclusion exclusion principle, the above mentioned rate of decay correlates with the length of the overlap of sets of the form $(\phi_{\om}F_i)_{\e}$, where $\om\in\Sigma^n$ for $n\in\mathbb N$ and $i\in\{0,1\}$, the value of the 1-st fractal curvature measure describes the distribution of the gaps. That is, the more equally spread the gaps are over the fractal, the greater is its fractal curvature measure.
Analogously, the value of the 0-th fractal curvature measure can be interpreted as the ``local fractal number of boundary points" or ``local fractal Euler number". For further details on the geometric interpretation in higher dimensions, we refer to \cite{WinterLlorente}, \cite{Winter_thesis} and \cite{Diplomarbeit}.

\subsection{Self-Conformal Sets and the Shift-Space}\label{sec:selfconfshift}

Let $X\subset\mathbb R$ be a nonempty compact interval. We call $\Phi\defeq\{\phi_i\colon X\to X\mid i\in\{1,\ldots,N\}\}$ a \emph{conformal iterated function system (cIFS)} acting on $X$, provided $N\geq 2$ and $\phi_1,\ldots,\phi_N$ are differentiable contractions with $\alpha$-Hölder continuous derivatives $\phi_1',\ldots,\phi_N'$, $\alpha>0$, where $\lvert\phi_1'\rvert,\ldots,\lvert\phi_N'\rvert$ shall be bounded away from both 0 and 1. 
The cIFS $\Phi\defeq\{\phi_1,\ldots,\phi_N\}$ is said to satisfy the \emph{open set condition (OSC)} if there exists a nonempty open bounded set $O\subset\mathbb R$ such that $\bigcup_{i=1}^{N}\phi_i(O)\subseteq O$ and $\phi_i(O)\cap\phi_j(O)=\emptyset$ for $i,j\in\{1,\ldots,N\},\ i\neq j$. $O$ then is called a \textit{feasible open set}. Note, that in our context conformality in dimension one just means that the derivatives are Hölder continuous. This extra condition can be dropped in higher dimensions.

\begin{defn}\label{def:selfconformal}
	We call the unique nonempty compact invariant set $F$ of a cIFS $\Phi$ the \emph{self-conformal} set associated with $\Phi$.
\end{defn}

\begin{rmk}
	One easily verifies that our definition of a cIFS coincides with the definition of a finite conformal iterated function system in $\mathbb R$ given in \cite{MauldinUrbanski}.
\end{rmk}
\begin{proposition}\label{intervalfractal}
	Let $\Phi$ be a cIFS acting on $X$ which satisfies the OSC with feasible open set $\textup{int} X$ and let $F$ be the self-conformal set associated with $\Phi$. Then $F$ is either a nonempty compact interval or has zero one-dimensional Lebesgue measure. 
\end{proposition}
\begin{proof}
	Let $\Phi\defeq\{\phi_1,\ldots,\phi_N\}$, define $X\eqdef [a,b]$ and assume without loss of generality that $\phi_1,\ldots,\phi_N$ are ordered such that $\phi_1(a)<\phi_2(a)<\ldots<\phi_N(a)$.
	If $\phi_i([a,b])\cap \phi_{i+1}([a,b])\neq\emptyset$ for all $i\in\{1,\ldots,N-1\}$, then clearly $F=[a,b]$.
	Now assume that there exists an $i\in\{1,\ldots,N-1\}$ such that $\phi_i([a,b])\cap \phi_{i+1}([a,b])=\emptyset$. Then Proposition 4.4 of \cite{MauldinUrbanski} gives that $F$ has zero Lebesgue measure.
\end{proof}

It turns out to be useful to view a self-conformal set on a symbolic level. For the following, we fix a cIFS $\Phi\defeq\{\phi_1,\ldots,\phi_N\}$ and let $F$ denote the self-conformal set associated with $\Phi$. We introduce the \emph{shift-space} $(\Sigma^{\infty},\sigma)$ as follows.

Let $\Sigma\defeq\{1,\ldots,N\}$ denote the \emph{alphabet} and $\Sigma^n$ the set of words of length $n\in\mathbb N$ over $\Sigma$ and denote by $\Sigma^*\defeq\bigcup_{n\in\mathbb N_0} \Sigma^n$ the set of all finite words over $\Sigma$ including the empty word $\emptyset$. 
Further, we call the set $\Sigma^{\infty}$ of infinite words over $\Sigma$ the \emph{code space}. 
The \emph{shift-map} is then defined as the map $\sigma\colon\Sigma^*\cup\Sigma^{\infty}\to\Sigma^*\cup\Sigma^{\infty}$ given by $\sigma(\om)\defeq \emptyset$ for $\om\in\{\emptyset\}\cup\Sigma^1$, $\sigma(\om_1\cdots\om_n)\defeq\om_2\cdots\om_n\in\Sigma^{n-1}$ for $\om_1\cdots\om_n\in\Sigma^n$, where $n\geq 2$ and 
$\sigma(\om_1\om_2\cdots)\defeq\om_2\om_3\cdots\in\Sigma^{\infty}$ for $\om_1\om_2\cdots\in\Sigma^{\infty}$.
For a finite word $\om\in\Sigma^*$ we let $n(\om)$ denote its length.\\

Note that $\Sigma^{\infty}$ gives a coding of the self-conformal set $F$ as can be seen as follows.
For $\om=\om_1\cdots\om_n\in\Sigma^*$ we set $\phi_{\om}\defeq\phi_{\om_1}\circ\cdots\circ\phi_{\om_n}$ and define $\phi_{\emptyset}\defeq\id\vert_{X}$ to be the identity map on $X$. For $\om=\om_1\om_2\cdots\in\Sigma^{\infty}$ and $n\in\mathbb N$ we denote the initial word by $\om\vert_n\defeq\om_1\om_2\cdots\om_n$. For each $\om=\om_1\om_2\cdots\in\Sigma^{\infty}$ the intersection $\bigcap_{n\in\mathbb N}\phi_{\om\vert_n}(X)$ contains exactly one point $x_{\om}\in F$ and gives rise to a surjection $\bij\colon{\Sigma^{\infty}}\to F,\ \om\mapsto x_{\om}$ which we call the \emph{natural code map}. 
Let $\Fun$ denote the set of points of $F$ which have a unique pre-image under $\bij$. Because of the open set condition $F\setminus\Fun$ is at most countable and thus $\Fun$ is nonempty. Moreover, $x\in\Fun$ implies $\phi_i(x)\in\Fun$ for all $i\in\Sigma$. 
The map $\bij$ allows to view points in $\Fun$ as infinite words and vice versa.
In order to have neater notation, we are often going to omit the map $\bij$. 
For example, by $\phi_{\om}(\omneu)$ we actually mean $\phi_{\om}(\bij(\omneu))$ for $\om\in\Sigma^*$ and $\omneu\in\Sigma^{\infty}$.\\

One of the key properties of a cIFS is the bounded distortion property. 
Our results require the following refinement of this property, which we could not find in this precise form in the literature. Therefore, we give a short proof of this statement. 
For a set $Y\subset\mathbb R$ let $\langle Y\rangle$ denote the convex hull of $Y$.

\begin{lem}[Bounded Distortion]\label{bd}
  There exists a sequence $(\bd_n)_{n\in\mathbb N}$ with $\bd_n>0$ for all $n\in\mathbb N$ and $\lim_{n\to\infty}\bd_n=1$ such that for all $\om,\omneu\in\Sigma^*$ and $x,y\in\langle\phi_{\om}F\rangle$ we have
  \begin{eqnarray*}
    \bd_{n(\om)}^{-1}\leq\frac{\lvert\phi_{\omneu}'(x)\rvert}{\lvert\phi_{\omneu}'(y)\rvert}\leq\bd_{n(\om)}.
  \end{eqnarray*}
\end{lem}

\begin{proof}
Fix $\om\in\Sigma^n$ and let $x,y\in\langle\phi_{\om}F\rangle$ and $\omneu=\omneu_1\cdots\omneu_{n(\omneu)}\in\Sigma^*$ be arbitrarily chosen. Then
\[
\frac{\lvert\phi_{\omneu}'(x)\rvert}{\lvert\phi_{\omneu}'(y)\rvert}
\leq  \exp\big(\sum_{k=1}^{n(\omneu)}\underbrace{\big{\lvert}\ln\lvert\phi_{\omneu_k}'(\phi_{\sigma^k\omneu}(x))\rvert-\ln\lvert\phi_{\omneu_k}'(\phi_{\sigma^k\omneu}(y))\rvert\big{\rvert}}_{\eqdef A_k}\big).
\]

Since $\lvert\phi'_i\rvert$ is $\alpha$-Hölder continuous and bounded away from 0, it follows that $\ln\lvert\phi'_i\rvert$ is $\alpha$-Hölder continuous for each $i\in\{1,\ldots,N\}$. 
Let $c_i$ be the corresponding Hölder constant and set $c\defeq\max_{i\in\{1,\ldots,N\}}c_i$. Moreover, let $r<1$ be a common upper bound for the contraction ratios of the maps $\phi_1,\ldots,\phi_N$. Without loss of generality we assume that $F\subseteq[0,1]$. Then we have
\[
A_k\leq c\lvert\phi_{\sigma^k\omneu}(x)-\phi_{\sigma^k\omneu}(y)\rvert^{\alpha}
\leq c\cdot\left( r^{n(u)-k}\lvert x-y\rvert\right)^{\alpha}
\]
and thus
\[
\sum_{k=1}^{n(u)}A_k
\leq \frac{c}{1-r^{\alpha}}\lvert x-y\rvert^{\alpha}
\leq \frac{c}{1-r^{\alpha}}\max_{\om\in\Sigma^n}\sup_{x,y\in\langle\phi_{\om}F\rangle}\lvert x-y\rvert^{\alpha}
\eqdef \tilde{\bd}_n.
\]
Since $\tilde{\bd}_n$ converges to 0 as $n\to\infty$, $\bd_n\defeq\exp(\tilde{\bd}_n)$ converges to 1 as $n\to\infty$.
The estimate for the lower bound can be obtained by just interchanging the roles of $x$ and $y$.	
\end{proof}

\subsection{Perron-Frobenius Theory and the Geometric Potential Function}\label{sec:PF}

In order to provide the necessary background to define the constants in our main statement and also to set up the tools needed in the proofs we now recall some facts from the Perron-Frobenius theory. 
For $f\in\mathcal C(\Sigma^{\infty})$, $0<\alpha<1$ and $n\in\mathbb N$ define 
\begin{eqnarray*}
  \text{var}_n(f)&\defeq&\sup\{\lvert f(\om)-f(\omneu)\rvert\mid \om, \omneu\in\Sigma^{\infty}\ \text{and}\ \om_i=\omneu_i\ \text{for all}\ i\in\{1,\ldots,n\}\},\\
  \lvert f\rvert_{\alpha}&\defeq&\sup_{n\geq 1}\frac{\text{var}_n(f)}{\alpha^{n-1}}\ \qquad\text{and}\\
  \mathcal F_{\alpha}(\Sigma^{\infty})&\defeq&\{f\in\mathcal C(\Sigma^{\infty})\mid \lvert f\rvert_{\alpha}<\infty\}.
\end{eqnarray*}
Elements of $\mathcal F_{\alpha}(\Sigma^{\infty})$ are called \emph{$\alpha$-Hölder continuous} functions on $\Sigma^{\infty}$.
For $f\in\mathcal{C}(\Sigma^{\infty})$ define the \emph{Perron-Frobenius operator} $\mathcal L_f\colon\mathcal C(\Sigma^{\infty})\to\mathcal C(\Sigma^{\infty})$ by
\begin{eqnarray}\label{PerronFrobenius}
	\mathcal L_f \psi(x)\defeq\sum_{y\colon\sigma y=x} \ee^{f(y)}\psi(y)
\end{eqnarray}
for $x\in\Sigma^{\infty}$ and let $\mathcal{L}_{ f}^*$ be the dual of $\mathcal{L}_{ f}$ acting on the set of Borel probability measures on $\Sigma^{\infty}$.
By Theorem 2.16 and Corollary 2.17 of \cite{Walters_convergence} and Theorem 1.7 of \cite{Bowen_equilibrium}, for each real valued Hölder continuous $f\in\mathcal F_{\alpha}(\Sigma^{\infty})$, some $0<\alpha<1$, there exists a unique Borel probability measure $\nu_{ f}$ on $\Sigma^{\infty}$ such that $\mathcal L_{ f}^*\nu_{ f}=\eigenv_{ f}\nu_{ f}$ for some $\eigenv_f>0$. Moreover, $\eigenv_f$ is uniquely determined by this equation and satisfies $\eigenv_f=\exp(P(f))$. Here $P\colon\mathcal C(\Sigma^{\infty})\to\mathbb R$ denotes the \emph{topological pressure function} which for $\psi\in\mathcal C(\Sigma^{\infty})$ is defined by
\[
P(\psi)\defeq\lim_{n\to\infty}n^{-1}\ln\sum_{\om\in\Sigma^n}\exp\sup_{\omneu\in[\om]}\sum_{k=0}^{n-1}\psi\circ\sigma^k(\omneu),
\]
(see Lemma 1.20 of \cite{Bowen_equilibrium}), where $[\om]\defeq\{\omneu\in\Sigma^{\infty}\mid \omneu_i=\om_i\ \text{for}\ 1\leq i\leq n(\om)\}$ is the \emph{$\om$-cylinder set}.

Further, there exists a unique strictly positive eigenfunction $h_{f}$ of $\mathcal L_f$ satisfying $\mathcal L_f h_f=\eigenv_f h_f$. We take $h_{ f}$ to be normalised so that $\int h_{ f}\textup{d}\nu_{ f}=1$. By $\mu_f$ we denote the $\sigma$-invariant probability measure defined by $\frac{\textup{d}\mu_{ f}}{\textup{d}\nu_{ f}}=h_{ f}$. This is the unique $\sigma$-invariant Gibbs measure for the potential function $f$. 
Additionally, under some normalisation assumptions we have convergence of the iterates of the Perron-Frobenius operator to the projection onto its eigenfunction $\eigenf_f$. To be more precise we have
\begin{eqnarray}\label{eq:convergenceperron}
	\lim_{m\to\infty}\|\eigenv_f^{-m}\mathcal L_f^m\psi - \textstyle{\int}\psi\textup{d}\nu_f\cdot \eigenf_f\|=0\quad\forall\ \psi\in\mathcal C(\Sigma^{\infty}),
\end{eqnarray}
where $\|\cdot\|$ denotes the supremum-norm on $\mathcal C(\Sigma^{\infty})$.
	The results on the Perron-Frobenius operator quoted above originate mainly from the work of \cite{Ruelle_gas}.

A central object of our investigations is the geometric potential function associated with the cIFS $\Phi$ and its property of being lattice or nonlattice, which we now define.

\begin{defn}\label{def:nonlattice}
	Two functions $f_1,f_2\in\mathcal C(\Sigma^{\infty})$ are called \emph{cohomologous}, if there exists a $\psi\in\mathcal C(\Sigma^{\infty})$ such that $f_1-f_2=\psi-\psi\circ\sigma$. A function $f\in\mathcal C(\Sigma^{\infty})$ is said to be a \emph{lattice} function, if $f$ is cohomologous to a function whose range is contained in a discrete subgroup of $\mathbb R$. Otherwise, we say that $f$ is a \emph{nonlattice} function.
\end{defn}

The notion of being lattice or not carries over to $\Phi$ and its self-conformal set $F$ by considering the geometric potential function associated with $\Phi$:

\begin{defn}\label{def:nonlatticefractal}
  Fix a cIFS $\Phi\defeq\{\phi_1,\ldots,\phi_N\}$. Denote by $F$ the self-conformal set associated with $\Phi$ and let $\Sigma^{\infty}$ be the associated code space. Define the \emph{geometric potential function} to be the map $\xi\colon\Sigma^{\infty}\to\mathbb R$ given by $\xi(\om)\defeq -\ln\lvert\phi_{\om_1}'(\sigma\om)\rvert$ for $\om=\om_1\om_2\cdots\in\Sigma^{\infty}$. If $\xi$ is nonlattice, then we call $\Phi$ (and also $F$) \emph{nonlattice}. On the other hand, if $\xi$ is a lattice function, then we call $\Phi$ (and also $F$) \emph{lattice}.
\end{defn}

\begin{rmk}\label{rmk:alpha}
	The geometric potential function $\xi$ associated with a cIFS $\Phi\defeq\{\phi_1,\ldots,\phi_N\}$ satisfies $\xi\in\mathcal F_{\tilde{\alpha}}(\Sigma^{\infty})$ for some $\tilde{\alpha}\in(0,1)$. To see this, we let $r<1$ be a common upper bound for the contraction ratios of $\phi_1,\ldots,\phi_N$. Because of the $\alpha$-Hölder continuity of $\phi'_1,\ldots,\phi'_N$ we obtain that there exists a constant $c\in\mathbb R$ such that for every $n\in\mathbb N$ we have $\text{var}_n(\xi)\leq cr^{\alpha(n-1)}$. Thus, $\xi\in\mathcal F_{\tilde{\alpha}}(\Sigma^{\infty})$, where $\tilde{\alpha}\defeq r^{\alpha}\in(0,1)$.
\end{rmk}

For the geometric potential function $\xi\in\mathcal{C}(\Sigma^{\infty})$ it can be shown that the \emph{measure theoretical entropy} $\entro_{\mu_{-\mdim\xi}}$ of the shift-map $\sigma$ with respect to $\mu_{-\mdim \xi}$ is given by
\begin{eqnarray}\label{entropy}
	\entro_{\mu_{-\mdim\xi}}=\mdim\int_{\Sigma^{\infty}}\xi\textup{d}\mu_{-\mdim \xi},
\end{eqnarray}
where $\mdim$ denotes the Minkowski dimension of $F$.
This observation follows for example from the variational principle, Theorem 1.22 of \cite{Bowen_equilibrium} and the following result of \cite{Bedford} which will also be needed in the proof of \cref{curvatureresult}.

\begin{thm}\label{thBedford}
	The Minkowski as well as the Hausdorff dimension of $F$ is equal to the unique real number $t>0$ such that $P(-t\xi)=0$, where $P$ denotes the topological pressure function. 
\end{thm}

In what follows, we fix a cIFS $\Phi\defeq\{\phi_1,\ldots,\phi_N\}$ acting on $X$ and let $\alpha>0$ denote the common Hölder exponent of $\phi'_1,\ldots,\phi'_N$. By $\mdim$ we denote the Minkowski dimension of the self-conformal set $F$ associated with $\Phi$ and by $\xi$ its geometric potential function. We are going to show that the eigenfunction $\eigenf_{-\mdim\xi}$ of the Perron-Frobenius operator $\mathcal{L}_{-\mdim\xi}$, which is defined on the code space $\Sigma^{\infty}$, can be extended to an $\alpha$-Hölder continuous function on $X$. For this we let $\mathcal{C}(X)$ denote the set of real valued continuous functions on $X$ and define the operator $\widetilde{\mathcal L}\colon\mathcal C(X)\to\mathcal C(X)$ by
\[
	\tilde{\mathcal L}(g)\defeq\sum_{i=1}^{N}\lvert\phi'_i\rvert^{\mdim}\cdot g\circ\phi_i.
\]
We remark that $\tilde{\mathcal L}$ acts continuously on $\mathcal C(\Sigma^{\infty})$ and that $\tilde{\mathcal L}$ is an extended version of the Perron-Frobenius operator given in (\ref{PerronFrobenius}) to functions which are defined on $X$. 
We let $\mathcal F_{\alpha}(X)$ denote the set of real valued $\alpha$-Hölder continuous functions on $X$.
\begin{thm}\label{UrbanskiFortsetzung}
	Let $\nu$ be the $\mdim$-conformal measure and $\xi$ the geometric potential function associated with the cIFS $\Phi\defeq\{\phi_1,\ldots,\phi_N\}$. Assume that $\Phi$ satisfies the OSC with feasible open set $\textup{int}X$. Let $F$ denote the self-conformal set associated with $\Phi$ and let $\mdim$ be its Minkowski dimension. Denote by $\bij$ the natural code map and by $\alpha$ the Hölder exponent of the functions $\phi'_1,\ldots,\phi'_N$. Then there exists a unique $\eigenf\in\mathcal{F}_{\alpha}(X)$ such that 
	\[
		\tilde{\mathcal L}\eigenf=\eigenf,\quad \int\eigenf\textup{d}\nu=1\quad\text{and}\quad \eigenf\vert_F\circ\bij=\eigenf_{-\mdim\xi},
	\]
	where $\eigenf_{-\mdim\xi}\in\mathcal C(\Sigma^{\infty})$ is the unique eigenfunction of $\mathcal L_{-\mdim\xi}$ corresponding to the eigenvalue 1.
\end{thm}

\begin{proof}
	We let $1$ denote the constant one-function on $X$. By Lemma 6.1.1 of \cite{Urbanski_Buch} the sequence $(\tilde{\mathcal L}^n(1))_{n\in\mathbb N}$ is uniformly bounded and equicontinuous and thus so is the sequence $(n^{-1}\sum_{i=0}^{n-1}\tilde{\mathcal L}^i(1))_{n\in\mathbb N}$. Therefore, by Arzelà-Ascoli, the sequence of averages exhibits an accumulation point which we denote by $\eigenf$.
	Obviously $\tilde{\mathcal L}\eigenf=\eigenf$ and $\int\eigenf\textup{d}\nu=1$.
	
	In order to show that $\eigenf\in\mathcal{F}_{\alpha}(X)$, it suffices to show that $f_n\defeq n^{-1}\sum_{i=0}^{n-1}\tilde{\mathcal L}^i(1)$ is $\alpha$-Hölder continuous for every $n\in\mathbb N$ and that the Hölder constants are uniformly bounded. For that we let $x,y\in X$.
	\begin{eqnarray*}
		&&\left\lvert f_n(x)-f_n(y)\right\rvert
		= \left\lvert n^{-1}\sum_{i=0}^{n-1}\sum_{\om\in\Sigma^i}\lvert\phi'_{\om}(x)\rvert^{\mdim}-\lvert\phi'_{\om}(y)\rvert^{\mdim}\right\rvert\\
		&&\quad\leq n^{-1}\sum_{i=0}^{n-1}\sum_{\om\in\Sigma^i}\left\lvert \exp\left(\mdim\sum_{k=1}^i\ln\lvert\phi'_{\om_k}(\phi_{\sigma^k\om}x)\rvert\right) -\exp\left(\mdim\sum_{k=1}^i\ln\lvert\phi'_{\om_k}(\phi_{\sigma^k\om}y)\rvert\right)\right\rvert.
	\end{eqnarray*}
	By hypotheses, $\ln\lvert\phi'_i\rvert$ is $\alpha$-Hölder continuous for every $i\in\{1,\ldots,N\}$. Let $c_1,\ldots,c_N$ denote the respective Hölder constants of $\ln\lvert\phi'_1\rvert,\ldots,\ln\lvert\phi'_N\rvert$, set $c\defeq\max_{i=1,\ldots,N}c_i$ and let $r<1$ be a common upper bound for the contraction ratios of $\phi_1,\ldots,\phi_N$. Applying the Mean Value Theorem to $\exp$ and letting $\theta_{\om}$ denote the mean value corresponding to the $\om$-summand, we obtain the following set of inequalities.
	\begin{eqnarray*}
		\left\lvert f_n(x)-f_n(y)\right\rvert
		&\leq&n^{-1}\sum_{i=0}^{n-1}\sum_{\om\in\Sigma^i}\ee^{\theta_{\om}}\cdot \mdim\sum_{k=1}^i c\lvert\phi_{\sigma^k\om}x-\phi_{\sigma^k\om}y\rvert^{\alpha}\\
		&\leq&n^{-1}\sum_{i=0}^{n-1}\sum_{\om\in\Sigma^i}\ee^{\theta_{\om}}\cdot \frac{\mdim c}{1-r^{\alpha}}\lvert x-y\rvert^{\alpha}.				 
	\end{eqnarray*}
	Since $\theta_{\om}$ lies between $\ln\lvert\phi'_{\om}(x)\rvert^{\mdim}$ and $\ln\lvert\phi'_{\om}(y)\rvert^{\mdim}$, there exists a $\tilde{\theta}_{\om}\in\mathbb R$ such that $\lvert\phi'_{\om}(\tilde{\theta}_{\om})\rvert^{\mdim}=\ee^{\theta_{\om}}$. By definition of the $\mdim$-conformal measure it can be easily seen that $\lvert\phi'_{\om}(\tilde{\theta}_{\om})\rvert^{\mdim}\leq\bd_0\nu(\phi_{\om}F)$. Thus, 
	\begin{eqnarray*}
		\left\lvert f_n(x)-f_n(y)\right\rvert
		&\leq&\underbrace{\frac{\bd_0\mdim c}{1-r^{\alpha}}}_{\eqdef\tilde{c}}\lvert x-y\rvert^{\alpha}.
	\end{eqnarray*}
	Hence the Hölder constant of each function $f_n$ is bounded by $\tilde c$. The uniqueness of $\eigenf$ and $\eigenf\vert_F\circ\bij=\eigenf_{-\mdim\xi}$ follow from Theorem 6.1.2 of \cite{Urbanski_Buch}.
\end{proof}

\subsection{Renewal Theory and Geometric Measure Theory}

In the proof of \cref{curvatureresult} we are going to make use of a renewal theory argument for counting measures in symbolic dynamics. For this we first fix the following notation.

For a map $f\colon\Sigma^{\infty}\to\mathbb R$ and $n\in\mathbb N$ define the \emph{$n$-th ergodic sum} to be $S_n f\defeq\sum_{k=0}^{n-1} f\circ\sigma^k$ and $S_0 f\defeq 0$.
Moreover, we call a function $f_1\colon(0,\infty)\to\mathbb R$ \emph{asymptotic} to a function $f_2\colon(0,\infty)\to\mathbb R$ as $\e\to 0$, in symbols $f_1(\e)\sim f_2(\e)$ as $\e\to 0$, if $\lim_{\e\to 0}f_1(\e)/f_2(\e)=1$. Similarly, we say that $f_1$ is \emph{asymptotic} to $f_2$ as $t\to\infty$, in symbols $f_1(t)\sim f_2(t)$ as $t\to\infty$, if $\lim_{t\to \infty}f_1(t)/f_2(t)=1$.

The following proposition is a well-known fact which is for example stated in Proposition 2.1 of \cite{Lalley}.
\begin{proposition}\label{prop:sunique}
	Let $f\in\mathcal F_{\alpha}(\Sigma^{\infty})$ for some $0<\alpha<1$ be such that for some $n\geq 1$ the function $S_n f$ is strictly positive on $\Sigma^{\infty}$. Then there exists a unique $s>0$ such that 
	\begin{eqnarray}\label{sunique}
		\eigenv_{-sf}=1.
	\end{eqnarray}
\end{proposition}
The following two theorems play a crucial role in the proof of \cref{curvatureresult}. The first of the two theorems is Theorem 1 of \cite{Lalley}. The second one is a refinement and generalisation of Theorem 3 in \cite{Lalley} and hence we will give a proof.

\begin{proposition}[Lalley]\label{thmlalley}
	Assume that $f$ lies in $\mathcal F_{\alpha}(\Sigma^{\infty})$ for some $0<\alpha<1$, is nonlattice and such that for some $n\geq 1$ the function $S_n f$ is strictly positive. Let $g\in\mathcal F_{\alpha}(\Sigma^{\infty})$ be nonnegative but not identically zero and let $s>0$ be implicitly given by Equation (\ref{sunique}). Then we have that
	\[
		\sum_{n=0}^{\infty}\sum_{y\colon\sigma^n y=x}g(y)\mathds 1_{\{S_n f(y)\leq t\}\vp}
		\sim \frac{\int g\textup{d}\nu_{-sf}}{s\int f\textup{d}\mu_{-sf}} \eigenf_{-s f}(x)\ee^{s t}
	\]
	as $t\to\infty$ uniformly for $x\in\Sigma^{\infty}$.
\end{proposition}

For $b\in \mathbb R$, we denote by $\lceil b\rceil$ the smallest integer which is greater than or equal to $b$, by $\lfloor b\rfloor$ the greatest integer which is less than or equal to $b$, and by $\{b\}$ the fractional part of $b$, that is $\{b\}\defeq b-\lfloor b\rfloor$. 
\begin{thm}\label{thmlalleyneu}
	Assume that $f$ lies in $\mathcal F_{\alpha}(\Sigma^{\infty})$ for some $0<\alpha<1$ and that for some $n\geq 1$ the function $S_n f$ is strictly positive. Further assume that $f$ is lattice and let $\z,\psi\in\mathcal C(\Sigma^{\infty})$ denote functions which satisfy 
	\[
		f-\z=\psi-\psi\circ\sigma,
	\]
	where $\z$ is a function whose range is contained in a discrete subgroup of $\mathbb R$. Let $\aaa>0$ be maximal such that $\z(\Sigma^{\infty})\subseteq\aaa\mathbb Z$. Further, let $g\in\mathcal F_{\alpha}(\Sigma^{\infty})$ be nonnegative but not identically zero and $s>0$ be implicitly given by Equation (\ref{sunique}). Then we have that
	
	\begin{eqnarray}
		\sum_{n=0}^{\infty}\sum_{y\colon\sigma^n y=x}g(y)\mathds 1_{\{S_n f(y)\leq t\}\vp}
			\sim \frac{\aaa\eigenf_{-s\z}(x)\int g(y)\ee^{-s\aaa\left\lceil\frac{\psi(y)-\psi(x)}{\aaa}-\frac{t}{\aaa}\right\rceil}\textup{d}\nu_{-s\z}(y)}{\left(1-\ee^{-s\aaa}\right)\int\z\textup{d}\mu_{-s\z}}	\end{eqnarray}
as $t\to\infty$ uniformly for $x\in\Sigma^{\infty}$.
\end{thm}
\begin{rmk}
	\Cref{thmlalley,thmlalleyneu} are also valid in the more general situation of $(\Sigma^{\infty},\sigma)$ being a subshift of finite type. See also Theorem 3 of \cite{Lalley} where the exact asymptotic is not provided.
\end{rmk}

\begin{proof}[Proof of \cref{thmlalleyneu}]
	For the proof we first assume that $\aaa=1$, which implies that $\z$ is integer valued and not cohomologous to any function taking its values in a proper subgroup of $\mathbb Z$. 
	We first follow the lines of the proofs of Theorem 2 and Theorem 3 of \cite{Lalley} and then refine the last steps of the proof of Theorem 3 of \cite{Lalley} to obtain the exact asymptotics. 
	
	Lalley introduces the following functions for fixed $t\in\mathbb R$ and $x\in\Sigma^{\infty}$.
	\begin{eqnarray*}
		N_f(t,x)&\defeq&\sum_{n=0}^{\infty}\sum_{y\colon\sigma^n y=x}g(y)\mathds 1_{\{S_n f(y)\leq t\}\vp},\\
		N^*(t,x)&\defeq&N_f(t-\psi(x),x)
	\end{eqnarray*}
	and for $\beta\in[0,1)$ and $z\in\mathbb C$ the Fourier-Laplace transform
	\begin{eqnarray*}
		\hat{N}_{\beta}^*(z,x)\defeq\sum_{n=-\infty}^{\infty}\ee^{nz}N^*(n+\beta,x).
	\end{eqnarray*}
	It is easy to verify that $N_f(t,x)$ satisfies a renewal equation (see Equation (2.2) in \cite{Lalley})
	\begin{eqnarray*}
		N_f(t,x)=\sum_{y\colon\sigma y=x}N_f(t-f(y),y)+g(x)\mathds 1_{\{t\geq 0\}\vp}
	\end{eqnarray*}
	from which one can deduce that $\hat{N}_{\beta}^*$ satisfies the following equation.
	\begin{eqnarray}\label{Nbetahat}
		\hat{N}_{\beta}^*(z,x)=(I-\mathcal{L}_{z\z})^{-1}g(x)\frac{\ee^{z\left\lceil\phi(x)-\beta\right\rceil}}{1-\ee^z},		
	\end{eqnarray}
	where $I$ denotes the identity operator.
	We remark that Equation (\ref{Nbetahat}) differs slightly from the respective equation in \cite{Lalley}, in that Lalley obtains $z\left\lfloor\phi(x)+1-\beta\right\rfloor$ as the argument of the exponential, whereas our calculations result in $z\left\lceil\phi(x)-\beta\right\rceil$ being the right expression instead. 
	
	By arguments in the proof of Theorem 2 of \cite{Lalley} the function $z\mapsto (I-\mathcal{L}_{z\z})^{-1}g(x)$ is meromorphic in $\{z\in\mathbb C\mid 0\leq\text{Im}(z)\leq\pi,\ \text{Re}(z)<-s+\e\}$ for some $\e>0$ and the only singularity in this region is a simple pole at $z=-s$ with residue
	\begin{eqnarray*}
	\frac{\eigenf_{-s\z}(x)\int g\textup{d}\nu_{-s\z}}{\int\z\textup{d}\mu_{-s\z}}.
	\end{eqnarray*}
	Since $z\mapsto \ee^{z\left\lceil\psi(x)-\beta\right\rceil}$ and $z\mapsto(1-\ee^z)^{-1}$ are holomorphic in $\{z\in\mathbb C\mid\text{Re}(z)<0\}$ we deduce from this that $z\mapsto\hat{N}_{\beta}^*(z,x)$ is meromorphic in $\{z\in\mathbb C\mid 0\leq\text{Im}(z)\leq\pi,\ \text{Re}(z)<-s+\e\}$ for some $\e>0$ and that the only singularity in this region is a simple pole at $z=-s$ with residue
	\begin{eqnarray*}
\frac{\eigenf_{-s\z}(x)\int g(y)\ee^{-s\left\lceil\psi(y)-\beta\right\rceil}\textup{d}\nu_{-s\z}(y)}{(1-\ee^{-s})\int\z\textup{d}\mu_{-s\z}}\eqdef C(\beta,x).
	\end{eqnarray*}
	Now, again following the lines of the proof of Theorem 2 of \cite{Lalley}, it follows that
	\begin{eqnarray*}
		N^*(n+\beta,x)\sim C(\beta,x)\ee^{sn}
	\end{eqnarray*}	
	as $n\to\infty$ uniformly for $x\in\Sigma^{\infty}$. Thus for $t\in(0,\infty)$	
	\begin{eqnarray}
		N_f(t,x)
		&=& N_f(\underbrace{\left\lfloor\psi(x)+t\right\rfloor}_{\eqdef n}+\underbrace{\{\psi(x)+t\}}_{\eqdef\beta}-\psi(x),x)
		= N^*(n+\beta,x)\nonumber\\
		&\sim& C(\beta,x)\ee^{sn}
=\frac{\eigenf_{-s\z}(x)\int g(y)\ee^{-s\left\lceil\psi(y)-\psi(x)-t\right\rceil}\textup{d}\nu_{-s\z}(y)}{(1-\ee^{-s})\int\z\textup{d}\mu_{-s\z}}\label{Nftx}
	\end{eqnarray}
	as $n\to\infty$ uniformly for $x\in\Sigma^{\infty}$. 
This proves the case $\aaa=1$. 
	
	The case that $\aaa\neq 1$ is not covered in \cite{Lalley}. If $\aaa>0$ is arbitrary, then we consider the function $\aaa^{-1}f=\aaa^{-1}\z+\aaa^{-1}\psi-\left(\aaa^{-1}\psi\right)\circ\sigma$.
	Since by \Cref{prop:sunique}, $s>0$ satisfying Equation (\ref{sunique}) is the unique positive real number such that $\eigenv_{-sf}=1$, $\tilde{s}\defeq s\aaa$ is the unique positive real number satisfying $\eigenv_{-\tilde{s}\aaa^{-1}f}=1$. Therefore, Equation (\ref{Nftx}) implies
	\begin{eqnarray*}
		\sum_{n=0}^{\infty}\sum_{y\colon\sigma^n y=x}g(y)\mathds 1_{\{S_n f(y)\leq t\}\vp}
		&=& \sum_{n=0}^{\infty}\sum_{y\colon\sigma^n y=x}g(y)\mathds 1_{\{S_n \aaa^{-1}f(y)\leq t\aaa^{-1}\}\vp}\\
		&\sim& \frac{\eigenf_{-s\z}(x)\int g(y)\ee^{-s\aaa \left\lceil\frac{\psi(y)-\psi(x)}{\aaa}-\frac{t}{\aaa}\right\rceil}\textup{d}\nu_{-s\z}(y)}{(1-\ee^{-s\aaa})\int\aaa^{-1}\z\textup{d}\mu_{-s\z}}
	\end{eqnarray*}
	as $t\to\infty$ uniformly for $x\in\Sigma^{\infty}$.
\end{proof}

In view of the existence of the average fractal curvature measures the following corollary is essential.
\begin{cor}\label{thmlalleyaverage}
	Under the assumptions of \cref{thmlalleyneu}
	\begin{eqnarray*}
		\lim_{T\to\infty}T^{-1}\int_{0}^{T}\ee^{-st} \sum_{n=0}^{\infty}\sum_{y\colon\sigma^n y=x}g(y)\mathds 1_{\{S_n f(y)\leq t\}\vp} \textup{d}t
	\end{eqnarray*}
	exists and equals 
	\begin{eqnarray*}
\frac{\eigenf_{-sf}(x)\int g\textup{d}\nu_{-sf}}{s\int f\textup{d}\mu_{-sf}}.
	\end{eqnarray*}
\end{cor}
\begin{proof}
	First, observe that for two functions $f_1,\,f_2\colon(0,\infty)\to\mathbb R$ which satisfy $f_1(t)\sim f_2(t)$ as $t\to\infty$, the existence of $G_1\defeq\lim_{T\to\infty}T^{-1}\int_0^T f_1(t)\textup{d}t$ implies the existence of $G_2\defeq\lim_{T\to\infty}T^{-1}\int_0^T f_2(t)\textup{d}t$ and $G_1=G_2$. In view of \cref{thmlalleyneu}, we hence consider the function $\eta\colon[0,\infty)\to\mathbb R$ given by
	\begin{eqnarray*}
		\eta(t)\defeq \ee^{-st}\int_{\Sigma^{\infty}} g(y)\ee^{-s\aaa\left\lceil\frac{\psi(y)-\psi(x)}{\aaa}-\frac{t}{\aaa}\right\rceil}\textup{d}\nu_{-s\z}(y).
	\end{eqnarray*}	
	Since $\eta(t+\aaa)=\eta(t)$ for all $t\in(0,\infty)$, $\eta$ is periodic with period $\aaa$. As $\eta$ is moreover locally integrable, this implies 
	\begin{eqnarray*}
		\lim_{T\to\infty}T^{-1}\hspace{-0.1cm}\int_{0}^T\hspace{-0.1cm}\eta(t)\textup{d}t
		\hspace{-0.2cm}&=&\hspace{-0.2cm} \lim_{T\to\infty}T^{-1}\bigg(\sum_{k=0}^{\left\lfloor\aaa^{-1}T\right\rfloor-1}\hspace{-0.1cm}\int_{T-\aaa(k+1)}^{T-\aaa k}\hspace{-0.1cm}\eta(t)\textup{d}t+\int_{0}^{T-a\left\lfloor\aaa^{-1}T\right\rfloor}\hspace{-0.1cm}\eta(t)\textup{d}t\bigg)\\
		\hspace{-0.2cm}&=&\hspace{-0.2cm} \lim_{T\to\infty}T^{-1}\left\lfloor\aaa^{-1}T\right\rfloor\int_{0}^{\aaa}\eta(t)\textup{d}t
		= \aaa^{-1}\int_{0}^{\aaa}\eta(t)\textup{d}t.
	\end{eqnarray*}
	Applying Fubini's theorem yields
	\begin{eqnarray*}
		\int_{0}^{\aaa}\eta(t)\textup{d}t= \int_{\Sigma^{\infty}}\int_0^{\aaa}\ee^{-st}g(y)\ee^{-s\aaa\left\lceil\frac{\psi(y)-\psi(x)}{\aaa}-\frac{t}{\aaa}\right\rceil}\textup{d}t\textup{d}\nu_{-s\z}(y).
	\end{eqnarray*}
	Define $E(y)\defeq\aaa\{\aaa^{-1}\left(\psi(y)-\psi(x)\right)\}$. This is the unique real number in $[0,\aaa)$ such that $\aaa^{-1} \left(\psi(y)-\psi(x)-E(y)\right)\in\mathbb Z$. Since $\aaa^{-1}t\in[0,1)$ for $t\in[0,\aaa)$, we hence have
	\begin{eqnarray*}
		&&\hspace{-0.7cm}\int_{0}^{\aaa}\eta(t)\textup{d}t\\
		&&\hspace{-0.7cm}=\hspace{-0.1cm} \int_{\Sigma^{\infty}} \hspace{-0.18cm}\bigg(\hspace{-0.08cm}\int_0^{E(y)}\hspace{-0.15cm}\ee^{-st}g(y)\ee^{-s\aaa\left\lceil\frac{\psi(y)-\psi(x)}{\aaa}\right\rceil}\textup{d}t + \int_{E(y)}^{\aaa}\hspace{-0.15cm}\ee^{-st}g(y)\ee^{-s\aaa\left\lfloor\frac{\psi(y)-\psi(x)}{\aaa}\right\rfloor}\textup{d}t\hspace{-0.08cm}\bigg)\textup{d}\nu_{-s\z}(y)\\
		&&\hspace{-0.7cm}=\hspace{-0.1cm} \int_{\Sigma^{\infty}}\hspace{-0.25cm}\frac{g(y)}{s}\bigg( \hspace{-0.08cm}\ee^{-s\aaa\left\lceil\frac{\psi(y)-\psi(x)}{\aaa}\right\rceil}\big(1-\ee^{-sE(y)}\big)\hspace{-0.1cm} + \ee^{-s\aaa\left\lfloor\frac{\psi(y)-\psi(x)}{\aaa}\right\rfloor}\big(\ee^{-sE(y)}-\ee^{-s\aaa}\big)\hspace{-0.08cm}\bigg)\textup{d}\nu_{-s\z}(y)\\
		&&\hspace{-0.7cm}=\hspace{-0.1cm} \frac{1-\ee^{-s\aaa}}{s}\ee^{s\psi(x)}\int_{\Sigma^{\infty}} g(y)\ee^{-s\psi(y)}\textup{d}\nu_{-s\z}(y),
	\end{eqnarray*}
	where the last equality can be obtained by distinguishing the cases $E(y)\neq 0$ and $E(y)=0$, that is $\aaa^{-1}\left(\psi(y)-\psi(x)\right)\in\mathbb Z$.
	As by \cref{thmlalleyneu}
	\begin{eqnarray*}
		\ee^{-st}\sum_{n=0}^{\infty}\sum_{y\colon\sigma^n y=x}g(y)\mathds 1_{\{S_n f(y)\leq t\}\vp}
		\sim \frac{\aaa\eigenf_{-s\z}(x)}{\left(1-\ee^{-s\aaa}\right)\int\z\textup{d}\mu_{-s\z}} \eta(t)
	\end{eqnarray*}
	as $t\to\infty$ uniformly for $x\in\Sigma^{\infty}$, the entering remark of this proof now implies
	\begin{eqnarray*}
		&&\lim_{T\to\infty}T^{-1}\int_0^T	\ee^{-st}\sum_{n=0}^{\infty}\sum_{y\colon\sigma^n y=x}g(y)\mathds 1_{\{S_n f(y)\leq t\}\vp}\textup{d}t\\
		&&\qquad\qquad= \frac{\ee^{s\psi(x)}\eigenf_{-s\z}(x)}{s\int\z\textup{d}\mu_{-s\z}} \int_{\Sigma^{\infty}} g(y)\ee^{-s\psi(y)}\textup{d}\nu_{-s\z}(y).
	\end{eqnarray*}
	Finally, one easily verifies that $\ee^{s\psi}\eigenf_{-s\z}=\eigenf_{-sf}$, $\ee^{-s\psi}\textup{d}\nu_{-s\z}=\textup{d}\nu_{-sf}$ and $\int\z\textup{d}\mu_{-s\z}=\int f\textup{d}\mu_{-sf}$, which completes the proof.	
\end{proof}

In order to prove Part (\ref{conformalMinkowski:lattice}) of \cref{conformalMinkowski}, the following lemma which is closely related to \cref{thmlalleyneu} is needed.
\begin{lem}\label{lem:existanceB}
  Assume the conditions of \cref{thmlalleyneu} and fix a nonempty Borel set $B\subseteq\mathbb R$. For $x\in\Sigma^{\infty}$ define the function $\eta_B\colon(0,\infty)\to\mathbb R$ by 
  \[
  \eta_B(t)\defeq \ee^{-st}\int_{\Sigma^{\infty}}\mathds 1_{\psi^{-1}B\vp}(y)\ee^{-s\aaa\left\lceil\frac{\psi(y)-\psi(x)}{\aaa}-\frac{t}{\aaa}\right\rceil}\textup{d}\nu_{-s\z}(y).
  \]
  Then $\lim_{t\to\infty}\eta_B(t)$ exists if and only if for every $t\in[0,\aaa)$ we have
    \begin{eqnarray*}
      &&\sum_{n\in\mathbb Z}\ee^{-s\aaa n}\nu_{-s\z}\circ\psi^{-1}\big(B\cap[n\aaa,n\aaa+t)\big)\\
	&&\qquad\qquad=\frac{\ee^{st}-1}{\ee^{s\aaa}-1}\sum_{n\in\mathbb Z}\ee^{-s\aaa n}\nu_{-s\z}\circ\psi^{-1}\big(B\cap[n\aaa,(n+1)\aaa)\big).
    \end{eqnarray*}
\end{lem}
\begin{proof}
  First, note that the above sums are finite. $\eta_B$ is a periodic function with period $a$, meaning $\eta_B(t+\aaa)=\eta_B(t)$ for all $t\in(0,\infty)$. Therefore, $\lim_{t\to\infty}\eta_B(t)$ exists if and only if $\eta_B$ is a constant function. For $t\in[\psi(x),\psi(x)+\aaa)$ we have
    \begin{eqnarray*}
      &&\eta_B(t-\psi(x))
      = \ee^{s\psi(x)-st}\int_{\Sigma^{\infty}}\mathds 1_{\psi^{-1}B\vp}(y)\ee^{-s\aaa\left\lceil\frac{\psi(y)-t}{\aaa}\right\rceil}\textup{d}\nu_{-s\z}(y)\\
      &&\qquad= \ee^{s\psi(x)-st}\sum_{n\in\mathbb Z}\int_{n\aaa}^{(n+1)\aaa}\mathds 1_{B\vp}(y)\ee^{-s\aaa\left\lceil\frac{y-t}{\aaa}\right\rceil}\textup{d}\nu_{-s\z}\circ\psi^{-1}(y)\\
      &&\qquad= \ee^{s\psi(x)-st+s\aaa\left\lfloor\frac{t}{\aaa}\right\rfloor}\sum_{n\in\mathbb Z} \ee^{-s\aaa n}\bigg(\nu_{-s\z}\circ\psi^{-1}\big(B\cap [n\aaa,n\aaa+a\{\aaa^{-1}t\}]\big)\\
      &&\qquad\qquad\qquad+\ee^{-s\aaa}\nu_{-s\z}\circ\psi^{-1}\big(B\cap(n\aaa+a\{\aaa^{-1}t\},(n+1)\aaa)\big)\bigg)\\
      &&\qquad= \ee^{s\psi(x)-s\aaa\{\frac{t}{\aaa}\}}\sum_{n\in\mathbb Z} \ee^{-s\aaa n}\bigg((1-\ee^{-s\aaa})\nu_{-s\z}\circ\psi^{-1}\big(B\cap [n\aaa,n\aaa+a\{\aaa^{-1}t\}]\big)\\
      &&\qquad\qquad\qquad+\ee^{-s\aaa}\nu_{-s\z}\circ\psi^{-1}\big(B\cap[n\aaa,(n+1)\aaa)\big)\bigg).
    \end{eqnarray*}
    Thus, $\lim_{t\to\infty}\eta_B(t)$ exists if and only if there is a $\tilde{c}\in\mathbb R$ such that for every $t\in[0,\aaa)$
    \begin{eqnarray*}
      &&\hspace{-0.5cm}\sum_{n\in\mathbb Z} \ee^{-s\aaa n}\nu_{-s\z}\circ\psi^{-1}\big(B\cap [n\aaa,n\aaa+t]\big)\\
      &&\hspace{-0.5cm}\qquad = (1-\ee^{-s\aaa})^{-1}\bigg(\tilde{c}\ee^{st-s\psi(x)}-\ee^{-s\aaa}\sum_{n\in\mathbb Z} \ee^{-s\aaa n}\nu_{-s\z}\circ\psi^{-1}\big(B\cap[n\aaa,(n+1)\aaa)\big)\bigg).
    \end{eqnarray*}
    Taking the limit as $t$ tends to $\aaa$ we hence obtain 
    \[
    \tilde{c}=\ee^{s\psi(x)-s\aaa}\sum_{n\in\mathbb Z} \ee^{-s\aaa n}\nu_{-s\z}\circ\psi^{-1}\big(B\cap[n\aaa,(n+1)\aaa)\big)
    \]
    which proves the statement.
\end{proof}

Another important tool in the proofs of our results is a relationship between the 0-th and the 1-st (average) fractal curvature measures. In order to show that the existence of the 0-th fractal curvature measure implies the existence of the 1-st fractal curvature measure we use Corollary 3.2 of \cite{RatajWinter} which is a higher-dimensional and more general version of the following theorem.
\begin{thm}[Rataj, Winter]\label{lemWinterRataj}
	Let $Y\subset\mathbb R$ be a nonempty and compact set such that $\leb^1(Y)=0$. Then
	\[
		\liminf_{\e\to 0}\frac{\e^{\mdim}\leb^0(\partial Y_{\e})}{1-\delta}
		\hspace{-0.05cm}\leq\hspace{-0.05cm} \liminf_{\e\to 0}\e^{\mdim-1}\leb^1(Y_{\e})
		\hspace{-0.05cm}\leq\hspace{-0.05cm} \limsup_{\e\to 0}\e^{\mdim-1}\leb^1(Y_{\e})
		\hspace{-0.05cm}\leq\hspace{-0.05cm} \limsup_{\e\to 0}\frac{\e^{\mdim}\leb^0(\partial Y_{\e})}{1-\delta}.
	\]
\end{thm}

The proof is based on an interesting relationship between the derivative $\frac{\textup d}{\textup d\e}\leb^1(F_{\e})$ which exists Lebesgue almost everywhere and the quantity $\leb^0(\partial F_{\e})$ which was established in \cite{Stacho} for arbitrary bounded subsets of $\mathbb R^d$ and builds on the work of \cite{Kneser}. As this relationship is also of use for us, we state it in the form of Corollary 2.5 in \cite{RatajWinter}.

\begin{proposition}[Stachó]\label{Stachothm}
	Let $Y\subset\mathbb R$ be compact. Then the function $\e\mapsto\leb^1(Y_{\e})$ is differentiable for all but a countable number of $\e>0$ with differential
	\[
		\frac{\textup{d}}{\textup{d}\e}\leb^1(Y_{\e})=\leb^0(\partial Y_{\e}).
	\]
\end{proposition}

For the results on the average fractal curvature measures we use Part (ii) of Lemma 4.6 of \cite{RatajWinter} which is a higher-dimensional version of the next proposition.
\begin{proposition}[Rataj, Winter]\label{lemWinterRataj:average}
	Let $Y\subset\mathbb R$ be nonempty and compact and such that its Minkowski dimension $\mdim$ is strictly less than 1. If $\overline{\mathcal M}(Y)<\infty$, then 
	\begin{eqnarray*}
		\limsup_{T\searrow 0} \lvert\ln T\rvert^{-1}\int_T^1\e^{\mdim-2}\leb^1(Y_{\e})\textup{d}\e
		\hspace{-0.1cm}&=&\hspace{-0.1cm}(1-\mdim)^{-1}\limsup_{T\searrow 0} \lvert\ln T\rvert^{-1}\int_T^1\e^{\mdim-1}\leb^0(Y_{\e})\textup{d}\e,\\
		\liminf_{T\searrow 0} \lvert\ln T\rvert^{-1}\int_T^1\e^{\mdim-2}\leb^1(Y_{\e})\textup{d}\e
		\hspace{-0.1cm}&=&\hspace{-0.1cm}(1-\mdim)^{-1}\liminf_{T\searrow 0} \lvert\ln T\rvert^{-1}\int_T^1\e^{\mdim-1}\leb^0(Y_{\e})\textup{d}\e.
	\end{eqnarray*}
\end{proposition}

\section{Proofs of \cref{curvatureresult,conformalMinkowski}}\label{sec:proofs}
In this section we provide the proofs of \cref{curvatureresult,conformalMinkowski}. Since Parts (\ref{curvatureresult:average}) to (\ref{curvatureresult:lattice}) of \cref{curvatureresult} require different methods of proof, we are going to split this section into three subsections, each of which deals with one of these parts. But before subdividing the section, we make the following observations which are needed in the proofs of Parts (\ref{curvatureresult:average}) and (\ref{curvatureresult:nonlattice}) of \cref{curvatureresult}, and for \cref{conformalMinkowski}.

Without loss of generality we assume that $\{0,1\}\subset F\subseteq[0,1]$ as otherwise the result follows by rescaling. We start by giving the proof for the 0-th fractal curvature measure. For that we fix an $\e>0$ and consider the expression $\leb^0(\partial F_{\e}\cap(-\infty,\bb])/2$ for some $\bbb\in\mathbb R$. 
Since $\leb^0$ is the counting measure, $\leb^0(\partial F_{\e}\cap(-\infty,\bb])$ gives the number of endpoints of the connected components of $F_{\e}$ in $(-\infty,\bb]$. This number can be obtained by looking at how many complementary intervals of lengths greater than or equal to $2\e$ exist in $(-\infty,\bb]$:
\begin{eqnarray}\label{eq:lebesgue}
	\leb^0\big(\partial F_{\e}\cap(-\infty,\bb]\big)/2
	=\underbrace{\sum_{i=1}^{Q-1}\card\{\om\in\Sigma^*\mid L_{\om}^i\subseteq(-\infty,\bb],\ \Lomi\geq 2\e\}}_{\eqdef \Xi(\e)}+c_1/2,
\end{eqnarray}
where $c_1\in\{1,2,3\}$ depends on the value of $\bb$. 
Next, we need to find appropriate bounds for $\Xi(\e)$. For this, we choose an $m\in\mathbb N\cup\{0\}$ such that for all $\om\in\Sigma^m$ all \main\ gaps $L_{\om}^1,\ldots,L_{\om}^{\Q-1}$ of the sets $\phi_{\om}(F)$ are greater than or equal to $2\e$ and set
\begin{eqnarray*}
	\Xi_{\om}^i(\e)\defeq\card\{\omneu\in\Sigma^*\mid L_{\omneu\om}^i\subseteq(-\infty,\bb],\ \lvert L_{\omneu\om}^i\rvert\geq 2\e\}
\end{eqnarray*}
for each $\om\in\Sigma^m$ and $i\in\{1,\ldots,Q-1\}$.
We have the following connection.
\begin{eqnarray}\label{eq:Xisum}
	\sum_{i=1}^{\Q-1}\sum_{\om\in\Sigma^m}\Xi_{\om}^i(\e)
	\leq \Xi(\e)
	\leq\sum_{i=1}^{\Q-1}\sum_{\om\in\Sigma^m}\Xi_{\om}^i(\e)+\sum_{j=1}^{m}(\Q-1)\cdot N^{j-1}.
\end{eqnarray}
For the following, we fix $\bbb\in\mathbb R\setminus F$. Then $F\cap (-\infty,\bb]$ can be expressed as a finite union of sets of the form $\phi_{\kappa}F$, where $\kappa\in\Sigma^*$. To be more precise, let $l\in\mathbb N$ be minimal such that there exist $\kappa_1,\ldots,\kappa_l\in\Sigma^*$ satisfying
\begin{enumerate}
	\item $F\cap(-\infty,\bb]=\bigcup_{j=1}^l\phi_{\kappa_j}F$ and
	\item $\phi_{\kappa_i} F\cap\phi_{\kappa_j} F$ contains at most one point for all $i\neq j$, where $i,j\in\{1,\ldots,l\}$.
\end{enumerate}

Then for $Z\defeq \bigcup_{j=1}^l[\kappa_j]$ the function $\mathds 1_{Z\vp}$ is Hölder continuous. 
Making use of the existence of the bounded distortion constant $\bd_{n(\om)}$ of $\Phi$ on $\phi_{\om}F$ (see \cref{bd}), we can give estimates for $\Xi_{\om}^i(\e)$, namely for an arbitrary $x\in \Fun$ we have

\begin{eqnarray}
	\Xi_{\om}^i(\e)
	\leq \underbrace{\sum_{n=0}^{\infty}\sum_{\omneu\in\Sigma^n}\mathds 1_{Z\vp}(\omneu\om x)\mathds 1_{\{\lvert\phi_{\omneu}'(\phi_{\om}x)\rvert\cdot\bd_{n(\om)}\cdot\Lomi\geq2\e\}}\vp}_{\eqdef\overline{A}_{\om}^i(x,\e,Z)}+ \overline{c}_2(x,Z),\label{eq:Aomover}
\end{eqnarray}

where we need to insert the constant $\overline{c}_2(x,Z)$ because of the following reason. $L_{\omneu\om}^i\subseteq(-\infty,\bb]$ does not necessarily imply $\omneu\om x\in Z$ for an arbitrary $x\in\Fun$. However, if $n(\omneu)\geq\max_{j=1,\ldots,l}n(\kappa_j)$, either $[\omneu\om]\subseteq Z$ or $[\omneu\om]\cap Z=\emptyset$ obtains. Hence, there are only finitely many $\omneu\in\Sigma^*$ such that $L_{\omneu\om}^i\subseteq(-\infty,\bb]$ does not imply $\omneu\om x\in Z$ for all $x\in\Fun$. Letting $\overline{c}_2(x,Z)\in\mathbb R$ denote this finite number shows that Equation (\ref{eq:Aomover}) is true for all $\e>0$. Likewise, there exists a constant $\underline{c}_2(x,Z)\in\mathbb R$ such that for all $\e>0$
\begin{eqnarray}
	\Xi_{\om}^i(\e)
	\geq \underbrace{\sum_{n=0}^{\infty}\sum_{\omneu\in\Sigma^n} \mathds 1_{Z\vp}(\omneu\om x)\cdot\mathds 1_{\{\lvert\phi_{\omneu}'(\phi_{\om}x)\rvert\cdot\bd_{n(\om)}^{-1}\cdot\Lomi\geq2\e\}}\vp}_{\eqdef \underline{A}_{\om}^i(x,\e,Z)}-\underline{c}_2(x,Z).\label{eq:Aomunder}
\end{eqnarray}

Combining Equations (\ref{eq:lebesgue})-(\ref{eq:Aomunder}) we obtain that for all $m\in\mathbb N$ and $x\in\Fun$
\begin{eqnarray}
	\overline{C}_0^f(F,(-\infty,\bb])
	&\leq& \limsup_{\e\to 0}\e^{\mdim}\sum_{i=1}^{Q-1}\sum_{\om\in\Sigma^m} \overline{A}_{\om}^i(x,\e,Z)\quad\text{and}\label{eq:fcmoben}\\
	\underline{C}_0^f(F,(-\infty,\bb])
	&\geq& \liminf_{\e\to 0}\e^{\mdim}\sum_{i=1}^{Q-1}\sum_{\om\in\Sigma^m} \underline{A}_{\om}^i(x,\e,Z).\label{eq:fcmunten}	
\end{eqnarray}
In order to prove \cref{curvatureresult,conformalMinkowski} we want to apply \Cref{thmlalley,thmlalleyneu} to get asymptotics for both the expressions $\overline{A}_{\om}^i(x,\e,Z)$ and $\underline{A}_{\om}^i(x,\e,Z)$.
For this, note that
\begin{eqnarray}
	&&\sum_{\omneu\in\Sigma^n}\mathds 1_{Z\vp}(\omneu\om x)\cdot\mathds 1_{\{\lvert\phi_{\omneu}'(\phi_{\om}x)\rvert\cdot\bd_{n(\om)}^{\pm 1}\cdot\Lomi\geq2\e\}\vp}\nonumber\\
	&&\qquad\qquad\qquad= \sum_{y\colon\sigma^n y=\om x}\mathds 1_{Z\vp}(y)\cdot\mathds 1_{\{\sum_{k=1}^{n}-\ln\lvert\phi_{y_k}'(\sigma^k y)\rvert\leq-\ln\frac{2\e}{\Lomi\bd_{n(\om)}^{\pm 1}}\}\vp}\nonumber\\
	&&\qquad\qquad\qquad= \sum_{y\colon\sigma^n y=\om x}\mathds 1_{Z\vp}(y)\cdot\mathds 1_{\{S_n\xi(y)\leq-\ln\frac{2\e}{\Lomi\bd_{n(\om)}^{\pm 1}}\}\vp}.\label{eq:Snpm}
\end{eqnarray}

The hypotheses and \cref{rmk:alpha} imply that the geometric potential function $\xi$ is Hölder continuous and strictly positive. The unique $s>0$ for which $\eigenv_{-s\xi}=1$ is precisely the Minkowski dimension $\mdim$ of $F$, which results by combining the fact that $\eigenv_{-s\xi}=\exp(P(-s\xi))$ for each $s>0$ and \cref{thBedford}.

Before we distinguish between the lattice and nonlattice case and give the proof of \cref{curvatureresult}, we prove the following lemma, which is needed in the proofs of all three parts of \cref{curvatureresult}.

\begin{lem}\label{Upsilon}
	For an arbitrary $x\in\Sigma^{\infty}$ and $\Upsilon\in\mathbb R$ we have that
	\begin{enumerate}
		\item 
			$\Upsilon\leq\sum_{i=1}^{\Q-1}\sum_{\om\in\Sigma^m}\eigenf_{-\mdim\xi}(\om x)\big(\Lomi\bd_{m}\big)^{\mdim}$ for all $m\in\mathbb N$ implies
		\[
			\Upsilon\leq\liminf_{m\to\infty}\sum_{i=1}^{\Q-1}\sum_{\om\in\Sigma^m}\Lomi^{\mdim}.
		\]
		\item
			$\Upsilon\geq\sum_{i=1}^{\Q-1}\sum_{\om\in\Sigma^m}\eigenf_{-\mdim\xi}(\om x)\big(\Lomi\bd_{m}^{-1}\big)^{\mdim}$ for all $m\in\mathbb N$ implies
		\[
		\Upsilon\geq\limsup_{m\to\infty}\sum_{i=1}^{\Q-1}\sum_{\om\in\Sigma^m}\Lomi^{\mdim}.
		\]
	\end{enumerate}
\end{lem}

\begin{proof}
We are first going to approximate the eigenfunction $\eigenf_{-\mdim\xi}$ of the Perron-Frobenius operator $\mathcal L_{-\mdim\xi}$. For that we claim that $\mathcal L_{-\mdim\xi}^n 1(x)=\sum_{\omneu\in\Sigma^n}\lvert \phi_{\omneu}'(x)\rvert^{\mdim}$ for each $x\in\Sigma^{\infty}$ and $n\in\mathbb N$, where $1$ is the constant one-function. This can be easily seen by induction.
Since $\mathcal L_{-\mdim\xi}^n 1$ converges uniformly to the eigenfunction $\eigenf_{-\mdim\xi}$ when taking $n\to\infty$ (see Equation (\ref{eq:convergenceperron})) we have that 
\[
	\forall t>0\ \exists M\in\mathbb N\colon\forall n\geq M,\ \forall\, x\in\Sigma^{\infty}\colon \bigg{\lvert}\sum_{\omneu\in\Sigma^n}\lvert\phi_{\omneu}'(x)\rvert^{\mdim}-\eigenf_{-\mdim\xi}(x)\bigg{\rvert}<t.
\]

Furthermore, through \cref{bd} we know that 
\[
	\forall t'>0\ \exists M'\in\mathbb N\colon\forall m\geq M'\colon \lvert\bd_{m}-1\rvert<t'.
\]

Thus, for all $n\geq M$ and $m\geq M'$
\begin{eqnarray*}
	\Upsilon
	&\leq&\sum_{i=1}^{\Q-1}\sum_{\om\in\Sigma^m}\eigenf_{-\mdim\xi}(\om x)\left(\Lomi\bd_{m}\right)^{\mdim}\\
	&\leq& \sum_{i=1}^{\Q-1}\sum_{\om\in\Sigma^m}\left(\sum_{\omneu\in\Sigma^n}\lvert\phi_{\omneu}'(\phi_{\om}x)\rvert^{\mdim}+t\right)\left(\Lomi\bd_{m}\right)^{\mdim}\\
	&\leq& \sum_{i=1}^{\Q-1}\sum_{\om\in\Sigma^m}\sum_{\omneu\in\Sigma^n}\lvert\phi_{\omneu}(L_{\om}^i)\rvert^{\mdim}\bd_{m}^{2\mdim}+ t\sum_{i=1}^{\Q-1}\sum_{\om\in\Sigma^m}\left(\Lomi\bd_{m}\right)^{\mdim}\\
	&\leq&\left(1+t'\right)^{2\mdim}\sum_{i=1}^{\Q-1}\sum_{\om\in\Sigma^m}\sum_{\omneu\in\Sigma^n}\lvert L_{\omneu\om}^i\rvert^{\mdim}+ t(1+t')^{\mdim}\sum_{i=1}^{\Q-1}\sum_{\om\in\Sigma^m}\Lomi^{\mdim}
	\eqdef A_{m,n}
\end{eqnarray*}

Hence, for all $t,t'>0$
\begin{eqnarray*}
	&&\hspace{-0.7cm}\Upsilon
	\leq \liminf_{m\to\infty}\liminf_{n\to\infty} A_{m,n}\\
	&&\hspace{-0.7cm}\leq \left(1+t'\right)^{2\mdim}\liminf_{m\to\infty}\liminf_{n\to\infty}\sum_{i=1}^{\Q-1}\sum_{\om\in\Sigma^{m}}\sum_{\omneu\in\Sigma^n}\lvert L_{\omneu\om}^i\rvert^{\mdim}
	+t(1+t')^{\mdim}\limsup_{m\to\infty}\sum_{i=1}^{\Q-1}\sum_{\om\in\Sigma^m}\Lomi^{\mdim}.
\end{eqnarray*}

Because we have $\sum_{i=1}^{\Q-1}\sum_{\om\in\Sigma^m}\Lomi^{\mdim}\leq\sum_{i=1}^{\Q-1}\sum_{\om\in\Sigma^m}\|\phi'_{\om}\|^{\mdim}\eqdef a_m$, where $\|\cdot\|$ denotes the supremum-norm on $\mathcal C(X)$, and the sequence $(a_m)_{m\in\mathbb N}$ is bounded by Lemma 4.2.12 of \cite{Urbanski_Buch}, letting $t$ and $t'$ tend to zero then gives the assertion. 

The same arguments can be used to show that $\limsup_{m\to\infty}\sum_{i=1}^{\Q-1}\sum_{\om\in\Sigma^m}\Lomi^{\mdim}$ is a lower bound in the second case.
\end{proof}

\subsection{The Nonlattice Case}
\begin{proof}[Proof of Part {\rm(\ref{curvatureresult:nonlattice})} of \cref{curvatureresult}]
In this proof we fix the notation from the beginning of Section \ref{sec:proofs}. 

If $\mathds 1_{Z\vp}$ is identically zero, we immediately obtain $C_0^f(F,(-\infty,\bb])=0=\nu(F\cap(-\infty,\bb])$. Therefore, in the following, we assume that $\mathds 1_{Z\vp}$ is not identically zero. Since $\mathds 1_{Z\vp}$ is Hölder continuous, by combining Equations (\ref{eq:Aomover}), (\ref{eq:Aomunder}) and (\ref{eq:Snpm}), we see that \Cref{thmlalley} can be applied to  $\overline{A}_{\om}^i(x,\e,Z)$ and $\underline{A}_{\om}^i(x,\e,Z)$ giving the following asymptotics.

\begin{eqnarray}
	\overline{A}_{\om}^i(x,\e,Z)
	&\sim& \frac{\int\mathds 1_{Z\vp}\textup{d}\nu_{-\mdim\xi}}{\mdim\int\xi\textup{d}\mu_{-\mdim\xi}}\cdot\eigenf_{-\mdim\xi}(\om x)\cdot(2\e)^{-\mdim}\big(\Lomi\bd_{n(\om)}\big)^{\mdim}\quad\text{and}\label{eq:Aomoverasym}\\
	\underline{A}_{\om}^i(x,\e,Z)
	&\sim&  \frac{\int\mathds 1_{Z\vp}\textup{d}\nu_{-\mdim\xi}}{\mdim\int\xi\textup{d}\mu_{-\mdim\xi}}
\cdot\eigenf_{-\mdim\xi}(\om x)\cdot(2\e)^{-\mdim}\big(\Lomi\bd_{n(\om)}^{-1}\big)^{\mdim}\label{eq:Aomunderasym}
\end{eqnarray}
as $\e\to 0$ uniformly for $x\in\Sigma^{\infty}$.
We first put our focus on finding an upper bound for $\overline{C}_0^f(F,(-\infty,\bb])$. As in the statement of this theorem set $\entro_{\mu_{-\mdim\xi}}\defeq\mdim\int\xi\textup{d}\mu_{-\mdim\xi}$. Combining the Equations (\ref{eq:fcmoben}) and (\ref{eq:Aomoverasym}), we obtain for $x\in\Fun$ and all $m\in\mathbb N$
\begin{eqnarray*}
	\overline{C}_0^f(F,(-\infty,\bb])
	\leq \frac{2^{-\mdim}}{\entro_{\mu_{-\mdim\xi}}} \sum_{i=1}^{\Q-1}\sum_{\om\in\Sigma^m}\eigenf_{-\mdim \xi}(\om x)\left(\Lomi\bd_{m}\right)^{\mdim}\int_{\Sigma^{\infty}}\mathds 1_{Z\vp}\textup{d}\nu_{-\mdim\xi}.
\end{eqnarray*}

Now an application of \cref{Upsilon} implies
\begin{eqnarray}\label{Xisup:nonlattice}
	\overline{C}_0^f(F,(-\infty,\bb])
	\leq \frac{2^{-\mdim}}{\entro_{\mu_{-\mdim\xi}}} 
	\liminf_{m\to\infty}\sum_{i=1}^{\Q-1}\sum_{\om\in\Sigma^m}\Lomi^{\mdim}\int_{\Sigma^{\infty}}\mathds 1_{Z\vp}\textup{d}\nu_{-\mdim\xi}.
\end{eqnarray}

Analogously, one can conclude that
\begin{eqnarray}\label{Xiinf:nonlattice}
	\underline{C}_0^f(F,(-\infty,\bb])
	\geq \frac{2^{-\mdim}}{\entro_{\mu_{-\mdim\xi}}}\limsup_{m\to\infty}\sum_{i=1}^{\Q-1}\sum_{\om\in\Sigma^m}\Lomi^{\mdim}\int_{\Sigma^{\infty}}\mathds 1_{Z\vp}\textup{d}\nu_{-\mdim\xi}.
\end{eqnarray}

Combining the inequalities (\ref{Xisup:nonlattice}) and (\ref{Xiinf:nonlattice}) yields that all the limits occurring therein exist and are equal. Moreover, the $\mdim$-conformal measure introduced in (\ref{conformalmeasure}) and $\nu_{-\mdim\xi}$ satisfy the relation $\nu_{-\mdim\xi}(\mathds 1_{Z\vp})=\nu((-\infty,\bb])$. Therefore, 
\[
	C_0^f(F,(-\infty,\bb])
=\frac{2^{-\mdim}}{\entro_{\mu_{-\mdim\xi}}}\lim_{n\to\infty}\sum_{i=1}^{\Q-1}\sum_{\om\in\Sigma^n}\Lomi^{\mdim}\cdot \nu(F\cap(-\infty,\bb])
\]
holds for every $\bbb\in\mathbb R\setminus F$. As $\mathbb R\setminus F$ is dense in $\mathbb R$ the assertion concerning the 0-th fractal curvature measure follows.
The result on the 1-st fractal curvature measure now follows by applying \cref{lemWinterRataj}, as for every $\bbb\in\mathbb R\setminus F$ we have that $F_{\e}\cap(-\infty,\bb]=\left(F\cap(-\infty,\bb]\right)_{\e}$ for sufficiently small $\e>0$.
\end{proof}

\subsection{The Lattice Case}
This subsection addresses Part (\ref{curvatureresult:lattice}) of \cref{curvatureresult} and \cref{conformalMinkowski}.
For clarity, we are going to split the proof of Part {(\ref{curvatureresult:lattice})} of \cref{curvatureresult} into two parts. First, we are going to show the statement on the nonexistence of the fractal curvature measures. For that we need the following lemma.

\begin{lem}\label{periodic}
  Let $F$ denote a self-conformal set associated with the cIFS $\Phi\defeq\{\phi_1,\ldots,\phi_N\}$. Let $\mdim$ denote the Minkowski dimension of $F$ and let $B\subseteq\mathbb R$ denote a Borel set for which $F_{\e}\cap B=(F\cap B)_{\e}$ for all sufficiently small $\e>0$.
  Assume that there exists a positive, bounded, periodic and Borel-measurable function $f\colon\mathbb R^+\to\mathbb R^+$ which has the following properties.
  \begin{enumerate}
  \item\label{thislemcond1} $f$ is not equal to an almost everywhere constant function.
  \item\label{thislemcond2} There exists a sequence $(a_m)_{m\in\mathbb N}$, where $a_m>0$ for all $m\in\mathbb N$ and $a_m\to 1$ as $m\to\infty$ and a constant $c\in\mathbb R$ such that the following property is satisfied. For all $t>0$ and $m\in\mathbb N$ there exists an $M\in\mathbb N$ such that for all $T\geq M$
    \begin{eqnarray}\label{lem:nonexist}
      &&(1-t)a_m^{-\mdim}f(T-\ln a_m)-c\ee^{-\mdim T}\nonumber\\
      &&\qquad\quad\leq \ee^{-\mdim T}\leb^0(\partial F_{\ee^{-T}}\cap B)
      \leq (1+t)a_m^{\mdim}f(T+\ln a_m)+c\ee^{-\mdim T}.
    \end{eqnarray}
  \end{enumerate}
  Then for $k\in\{0,1\}$ we have
  \[
  \underline{C}_k^f(F,B)<\overline{C}_k^f(F,B).
  \]
\end{lem}
\begin{proof}
  We first cover the case $k=0$. Since $f$ is positive and not equal to an almost everywhere constant function, there exist $\tilde{T}_1,\tilde{T}_2>0$ such that $R\defeq f(\tilde{T}_2)/f(\tilde{T}_1)>1$. Choose $m\in\mathbb N$ so that $a_m^{2\mdim}<\sqrt{R}$ and choose $t>0$ such that $(1+t)/(1-t)<\sqrt{R}$. Then $\tilde{R}\defeq (1-t)a_m^{-\mdim}f(\tilde{T}_2)-(1+t)a_m^{\mdim}f(\tilde{T}_1)>0$. 
  By Condition (\ref{thislemcond2}) we can find an $M\in\mathbb N$ for these $t$ and $m$ such that for all $T\geq M$ Equation (\ref{lem:nonexist}) is satisfied.
  Because of the periodicity of $f$ we can find $T_1,T_2\geq M$ such that $f(\tilde{T}_1)=f(T_1+\ln a_m)$ and $f(\tilde{T}_2)=f(T_2-\ln a_m)$. Moreover, we can assume that $T_1,T_2$ are so large that $c\ee^{-\mdim T_1}+c\ee^{-\mdim T_2}\leq\tilde{R}/2$. Then
  \begin{eqnarray*}
    \ee^{-\mdim T_1}\leb^0(\partial F_{\ee^{-T_1}}\cap B)
    &\leq& (1+t) a_m^{\mdim}f(T_1+\ln a_m)+c\ee^{-\mdim T_1}\\
    &\leq& (1-t) a_m^{-\mdim} f(T_2-\ln a_m)-\tilde{R}/2-c\ee^{-\mdim T_2}\\
    &<& \ee^{-\mdim T_2}\leb^0(\partial F_{\ee^{-T_2}}\cap B).
  \end{eqnarray*}
  Because of the periodicity of $f$ this proves the case $k=0$. For $k=1$ observe that the function $g\colon\mathbb R^+\to\mathbb R^+$ defined by 
  \[
  g(T)\defeq\int_0^{\infty}f(s+T)\ee^{(\mdim-1)s}\textup{d}s
  \]
  is periodic. Also, $g$ is not a constant function. Since if it was, then $0=g(0)-g(T)$ for all $T\geq 0$. This would imply $\int_T^{\infty}f(s)\ee^{(\mdim-1)s}\textup{d}s=\ee^{(\mdim-1)T}\int_0^{\infty}f(s)\ee^{(\mdim-1)s}\textup{d}s$ for all $T\geq 0$. Differentiating with respect to $T$ would imply that $f$ itself is constant almost everywhere which is a contradiction. 
  Using that $F_{\e}\cap B=(F\cap B)_{\e}$ for sufficiently small $\e>0$ and Stachó's Theorem (\Cref{Stachothm}), we obtain for sufficiently large $T\geq 0$
  \begin{eqnarray*}
    &&\ee^{-T(\mdim-1)}\leb^1(F_{\ee^{-T}}\cap B)
    = \ee^{-T(\mdim-1)}\int_{T}^{\infty}\leb^0(\partial F_{\ee^{-s}}\cap B)\ee^{-s}\textup{d}s\\
    &&\qquad\qquad\leq \ee^{-T(\mdim-1)}(1+t) a_m^{\mdim}\int_{T}^{\infty} f(s+\ln a_m)\ee^{s(\mdim-1)}\textup{d}s+c\ee^{-T\mdim}\\
    &&\qquad\qquad= (1+t) a_m^{\mdim} g(T+\ln a_m)+c\ee^{-\mdim T}.
  \end{eqnarray*}
  Analogously, we obtain
  \begin{eqnarray*}
    \ee^{-T(\mdim-1)}\leb^1(F_{\ee^{-T}}\cap B)
    \geq (1-t) a_m^{-\mdim} g(T-\ln a_m)-c\ee^{-\mdim T}.
  \end{eqnarray*}
  Therefore, the same arguments which were used in the proof of the case $k=0$ imply that 
  \[
  \liminf_{\e\to 0}\e^{\mdim-1}\leb^1(F_{\e}\cap B)
  <\limsup_{\e\to 0}\e^{\mdim-1}\leb^1(F_{\e}\cap B).
  \]
\end{proof}

\begin{proof}[Proof of Part {\rm (\ref{curvatureresult:lattice})} of \cref{curvatureresult} (nonexistence)]
  We want to apply \cref{periodic} in order to show that there exists a Borel set $B\subseteq\mathbb R$ for which $\underline{C}_k^f(F,B)<\overline{C}_k^f(F,B)$ for $k\in\{0,1\}$ from which we then deduce that the fractal curvature measures do not exist. 
	For applying \cref{periodic} we first introduce a family $\Delta$ of nonempty Borel subsets of $\Sigma^{\infty}$. For every $\kappa\in\Delta$ we then construct a pair $(B(\kappa), f_{\kappa})$ which consists of a nonempty Borel set $B(\kappa)\subseteq\mathbb R$ satisfying $F_{\e}\cap B(\kappa)=(F\cap B(\kappa))_{\e}$ for all sufficiently small $\e>0$ and a positive bounded periodic Borel-measurable function $f_{\kappa}\colon\mathbb R^+\to\mathbb R^+$, such that Item (\ref{thislemcond2}) of \cref{periodic} is satisfied for $B=B(\kappa)$ and $f=f_{\kappa}$.
	Then, we show that there always exists a $\kappa\in\Delta$ for which $f_{\kappa}$ is not equal to an almost everywhere constant function, verifying Item (\ref{thislemcond1}) of \cref{periodic}.
	
	Let us begin by introducing the family $\Delta$. First, fix an $n\in\mathbb N\cup\{0\}$ and define
	\begin{eqnarray*}
		\hspace{-0.5cm}\Delta_n\defeq\Big\{\bigcup_{i=1}^l[\kappa^{(i)}]&\hspace{-0.2cm}\mid&\hspace{-0.3cm}\kappa^{(i)}\in\Sigma^{n},\, l\in\{1,\ldots,N^n\},\, \bigcup_{i=1}^l\langle\phi_{\kappa^{(i)}}F\rangle\ \text{is an interval},\,\\
		&&\hspace{-0.3cm}\bigcup_{i=1}^l\phi_{\kappa^{(i)}}F\cap\phi_{\om}F=\emptyset\ \text{for every}\ \om\in\Sigma^n\setminus\{\kappa^{(1)},\ldots,\kappa^{(l)}\}\Big\}.
	\end{eqnarray*}
	(Note that if the strong seperation condition was satisfied, then $\Delta_n=\{[\om]\mid\om\in\Sigma^n\}$.)
	We remark that the condition $\leb^1(F)=0$ implies that $\kappa\subsetneq\Sigma^{\infty}$ for every $\kappa\in\Delta_n$, whenever $n\in\mathbb N\setminus\{0\}$.
	Further, note that $\Delta_n\neq\emptyset$ for all $n\in\mathbb N$ because of the OSC and set $\Delta\defeq\bigcup_{n\in\mathbb N\cup\{0\}}\Delta_n$. 
	Now, fix an $n\in\mathbb N\cup\{0\}$ and a $\kappa=\bigcup_{i=1}^l[\kappa^{(i)}]\in\Delta_n$ and choose $\theta>0$ such that $\bigcup_{i=1}^l\langle\phi_{\kappa^{(i)}}F\rangle_{2\theta}\cap \phi_{\om}F=\emptyset$ for every $\om\in\Sigma^n\setminus\{\kappa^{(1)},\ldots,\kappa^{(l)}\}$.
	Then $B(\kappa)\defeq \bigcup_{i=1}^l\langle\phi_{\kappa^{(i)}}F\rangle_{\theta}$ is a nonempty Borel subset of $\mathbb R$ satisfying $F_{\e}\cap B(\kappa)=(F\cap B(\kappa))_{\e}$ for all $\e<\theta$.
	
	For constructing the function $f_{\kappa}$ fix an $m\in\mathbb N$ and choose $M\in\mathbb N$ so that $\ee^{-M}<\theta$ and that for every $\om\in\Sigma^m$ all main gaps of the sets $\phi_{\om}F$ which lie in $B(\kappa)$ are of length greater than or equal to $2\ee^{-M}$. Then for all $T\geq M$ we have
	\begin{eqnarray*}
    \leb^0\left(\partial F_{\ee^{-T}}\cap B(\kappa)\right)/2
    &=& \sum_{i=1}^{Q-1}\card\{\om\in\Sigma^*\mid L_{\om}^i\subseteq B(\kappa),\ \Lomi\geq2\ee^{-T}\}+1\\
    &\leq& \sum_{i=1}^{Q-1}\sum_{\om\in\Sigma^m}\Xi_{\om}^i(\ee^{-T})+\underbrace{\sum_{j=1}^{m-n-1}(Q-1)\cdot N^{j-1}+1}_{\eqdef c_m},
  \end{eqnarray*}
  where we agree that $\sum_{j=1}^{m-n-1}(Q-1)\cdot N^{j-1}=0$ if $m-n-1<1$ and where
  \begin{eqnarray*}
    \Xi_{\om}^i(\ee^{-T})\defeq\card\{\omneu\in\Sigma^*\mid L_{\omneu\om}^i\subseteq B(\kappa),\ \lvert L_{\omneu\om}^i\rvert\geq 2\ee^{-T}\}.
  \end{eqnarray*}
  Likewise
  \begin{eqnarray*}
    \leb^0\left(\partial F_{\ee^{-T}}\cap B(\kappa)\right)/2
    &\geq& \sum_{i=1}^{Q-1}\sum_{\om\in\Sigma^m}\Xi_{\om}^i(\ee^{-T}).
  \end{eqnarray*}

  Next, we use the lattice property of the cIFS $\Phi$. Since $\xi$ is a lattice function, there exist $\z,\psi\in\mathcal C(\Sigma^{\infty})$ such that
\[
	\xi-\z=\psi-\psi\circ\sigma
\]
and such that $\z$ is a function whose range is contained in a discrete subgroup of $\mathbb R$. Let $\aaa>0$ be the maximal real number such that $\z(\Sigma^{\infty})\subseteq\aaa\mathbb Z$. Recall from the beginning of Section \ref{sec:proofs} that the hypotheses and \cref{rmk:alpha} imply that $\xi$ is Hölder continuous and strictly positive and that the unique $s>0$ for which $\eigenv_{-s\xi}=1$ is the Minkowski dimension $\mdim$ of $F$. Moreover, recall that we assume without loss of generality that $\{0,1\}\subseteq F\subseteq[0,1]$. 
  We define $\tilde{\psi}\defeq\mdim^{-1}\ln\eigenf$, where $\eigenf\in\mathcal{F}_{\alpha}(X)$ is the positive function which is uniquely defined through \cref{UrbanskiFortsetzung}. $\tilde{\psi}$ satisfies the equation $\tilde{\psi}\circ\bij=\psi$ since $\eigenf$ satisfies
  \[
  \eigenf\circ\bij
  =\eigenf_{-\mdim\xi}=\frac{\textup{d}\mu_{-\mdim\xi}}{\textup{d}\nu_{-\mdim\xi}}
  =\frac{\textup{d}\mu_{-\mdim\z}}{\ee^{-\mdim\psi}\textup{d}\nu_{-\mdim\z}}=\ee^{\mdim\psi}.
  \]
  We define the function $\tilde{g}\colon [0,1]\to\mathbb R$ by $\tilde{g}(x)\defeq\int_0^x \ee^{\tilde{\psi}(y)}\textup{d}y/A$ for $x\in [0,1]$,where $A\defeq\int_0^1 \ee^{\tilde{\psi}(y)}\textup{d}y$. 
  As $\tilde{\psi}$ is $\alpha$-Hölder continuous, the Fundamental Theorem of Calculus implies that $\tilde{\psi}-\ln A=\ln \tilde{g}'$. 
  Moreover, the continuity of $\tilde{\psi}$ implies that $\tilde{\psi}$ is bounded on $[0,1]$. Therefore, $\tilde{g}'$ is bounded away from both 0 and $\infty$ and thus $\tilde{g}$ is invertible. Note that $\tilde{g}([0,1])=[0,1]$, set $g\defeq \tilde{g}^{-1}\colon [0,1]\to[0,1]$ and extend $g$ to a $\mathcal{C}^{1+\alpha}(\mathcal{U})$ function on an open neighbourhood $\mathcal U$ of $[0,1]$ such that $\lvert g'\rvert>0$ on $\mathcal U$. 
  Define $R_i\defeq g^{-1}\circ\phi_i\circ g$ for $i\in\{1,\ldots,N\}$ and $K\defeq g^{-1}(F)\subseteq[0,1]$. 
  Then setting $\bij_K\defeq g^{-1} \circ\bij$, we have
  \begin{eqnarray*}
    -\ln\lvert R'_{\om_1}(\bij_K\sigma\om)\rvert
    &=&-\ln \tilde{g}'(\phi_{\om_1}g\bij_K\sigma\om)-\ln\lvert\phi'_{\om_1}(g\bij_K\sigma\om)\rvert-\ln\lvert g'(\bij_K\sigma\om)\rvert\\
    &=&-\tilde{\psi}(\phi_{\om_1}\bij\sigma\om)+\ln A+\xi(\om)+\ln \tilde{g}'(g\bij_K\sigma\om)\\
    &=&\xi(\om)-\psi(\om)+\psi\circ\sigma(\om).
  \end{eqnarray*}
  Thus, $\z(\om)=-\ln\lvert R'_{\om_1}(\bij_K\sigma\om)\rvert$ for $\om\in\Sigma^{\infty}$. Since the range of $\z$ is contained in a discrete subgroup of $\mathbb R$ and $\xi$ and $\psi$ are bounded on $\Sigma^{\infty}$, $\z$ in fact takes a finite number of values.
  Moreover, $R'_i$ is $\alpha$-Hölder continuous. Therefore, there exists an $\tilde{M}\in\mathbb N$ such that for all $m\geq\tilde{M}$ we have that
\begin{eqnarray}\label{eq:rom}
  \forall\om\in\Sigma^m,\, \forall i\in\{1,\ldots,N\}\ \exists\, r_i^{\om}\in\mathbb R\colon\forall x\in[\om]\colon\ R'_i(\bij_K x)=r_i^{\om}.
\end{eqnarray}
Note, that $\ln r_i^{\om}\in\aaa\mathbb{Z}$ for all $i\in\{1,\ldots,N\}$ and $\om\in\Sigma^m$.
>From the fact that $\phi_1,\ldots,\phi_N$ are contractions and $g'$ is Hölder continuous and bounded away from $0$, one can deduce that there exists an iterate of $R\defeq\{R_1,\ldots,R_N\}$ which solemnly consists of contractions.
Without loss of generality we assume that $R_1,\ldots,R_N$ are contractions themselves. Then clearly, $R\defeq\{R_1,\ldots, R_N\}$ is a cIFS and therefore satisfies the bounded distortion property (see \cref{bd}). Let the associated sequence of bounded distortion constants be denoted by $(\bd_n)_{n\in\mathbb N}$.
Denote by $\tilde{L}^i$ the primary gaps of $K$ and by $\tilde{L}^i_{\om}$ the main gaps of $R_{\om} K$, where $i\in\{1,\ldots,N\}$ and $\om\in\Sigma^*$.
Let $c_g$ be the Hölder constant of $g'$ and let $k_g>0$ be such that $\lvert g'\rvert\geq k_g$ on $\mathcal U$. Since $K\subseteq[0,1]$ we have the following for all $x,y\in\langle R_{\om}K\rangle$, where $\om\in\Sigma^n$ and $n\in\mathbb N$. 
\begin{eqnarray}\label{eq:bdg}
  \left\lvert\frac{g'(x)}{g'(y)}\right\rvert
  \leq \left\lvert\frac{g'(x)-g'(y)}{g'(y)}\right\rvert+1
  \leq \frac{c_g\lvert x-y\rvert^{\alpha}}{k_g}+1
  \leq \max_{\om\in\Sigma^n}\frac{c_g\langle R_{\om}K\rangle^{\alpha}}{k_g}+1\eqdef \bdneu_n.\hphantom{..}
\end{eqnarray} 
Clearly, $\bdneu_n\to 1$ as $n\to\infty$.
Take $m\geq\tilde{M}$ and for $\om\in\Sigma^m$ write $\om=\om''\om'$, where $\om''\in\Sigma^{m-\tilde{M}}$ and $\om'\in\Sigma^{\tilde{M}}$.  
Combining Equations (\ref{eq:rom}) and (\ref{eq:bdg}) we now obtain that for $\omneu\in\Sigma^*$ and $i\in\{1,\ldots,Q-1\}$ we have for an arbitrary $x\in K^{\text{unique}}$
\begin{eqnarray*}
  \lvert L_{\omneu\om}^{i}\rvert
  &=& \lvert g\tilde{L}_{\omneu\om}^{i}\rvert
  \leq \lvert g'(R_{\omneu\om}x)\rvert\bdneu_{m}\lvert R'_{\omneu}(R_{\om}x)\rvert\bd_{m}\lvert\tilde{L}_{\om}^i\rvert\\ 
  &=& \lvert (g\circ R_{\omneu})'(R_{\om}x)\rvert\bdneu_{m}\bd_{m}\lvert\tilde{L}_{\om}^i\rvert
  = \lvert (\phi_{\omneu}\circ g)'(R_{\om}x)\rvert\bdneu_{m}\bd_{m}\lvert R_{\om''}\tilde{L}_{\om'}^i\rvert\\ 
  &\leq& \lvert \phi'_{\omneu}(gR_{\om}x)\rvert\lvert g'(R_{\om}x)\rvert\bdneu_{m}\bd_{m} r_{\om''}^{\om'}\bd_{\tilde{M}}\lvert\tilde{L}_{\om'}^i\rvert\\
  &=&\exp\Big(-S_{n(\omneu)}\xi(\omneu\om x)-\psi(\om x)+\ln(\underbrace{A\bdneu_{m}\bd_{m}\bd_{\tilde{M}}}_{\eqdef d_m} r_{\om''}^{\om'}\lvert\tilde{L}_{\om'}^i\rvert)\Big).
\end{eqnarray*}
Therefore, for $x\in K^{\text{unique}}$, $m\geq\max\{M,\tilde{M}\}$ and $\om\in\Sigma^m$
\begin{eqnarray*}
&&\Xi_{\om}^i(\ee^{-T})
\leq\card\{\omneu\in\Sigma^*\mid L_{\omneu\om}^i\subseteq B(\kappa),\\
&&\hphantom{\Xi_{\om}^i(\ee^{-T})\leq\card\{\omneu\in\Sigma^*\mid } S_{n(\omneu)}\xi(\omneu\om x)\leq -\ln(2\ee^{-T})+\ln(d_{m} r_{\om''}^{\om'}\lvert\tilde{L}_{\om'}^i\rvert)-\psi(\om x)\}.
\end{eqnarray*}
By construction we have $\mathds 1_{\kappa}\in\mathcal{F}_{\alpha}(\Sigma^{\infty})$.
Recalling that $\aaa>0$ denotes the maximal real number for which $\z(\Sigma^{\infty})\subseteq\aaa\mathbb Z$, an application of \cref{thmlalleyneu} hence yields
	\begin{eqnarray}
    &&\hspace{-1.55cm}\leb^0(\partial F_{\ee^{-T}}\cap B(\kappa))/2-c_m\nonumber\\
    &&\hspace{-1.55cm}\leq \sum_{i=1}^{Q-1}\sum_{\om\in\Sigma^{m}}\sum_{n=0}^{\infty}\sum_{y\colon\sigma^n y=\om x}\mathds 1_{\kappa\vp}(y)\cdot\mathds 1_{\{S_{n}\xi(y)\leq -\ln(2\ee^{-T})+\ln(d_{m}r_{\om''}^{\om'}\lvert\tilde{L}_{\om'}^i\rvert)-\psi(\om x)\}\vp}\nonumber\\
    &&\hspace{-1.55cm}\sim\sum_{i=1}^{Q-1}\hspace{-0.05cm}\sum_{\om\in\Sigma^{m}} \hspace{-0.1cm}\frac{\aaa\eigenf_{-\mdim\z}(\om x)\int\mathds 1_{\kappa\vp}(y)\ee^{-\mdim\aaa\left\lceil\hspace{-0.05cm}\frac{\psi(y)-\psi(\om x)}{\aaa}+\frac{1}{\aaa}\ln{\frac{2\ee^{-T}}{d_{m}r_{\om''}^{\om'}\lvert\tilde{L}_{\om'}^i\rvert}}+\frac{\psi(\om x)}{\aaa}\hspace{-0.05cm}\right\rceil}\textup{d}\nu_{-\mdim\z}(y)}{\left(1-\ee^{-\mdim\aaa}\right)\int\z\textup{d}\mu_{-\mdim\z}}.	\label{long}
  \end{eqnarray}
  
  Define $W\defeq\aaa\left(1-\ee^{-\mdim\aaa}\right)^{-1}\left(\int\z\textup{d}\mu_{-\mdim\z}\right)^{-1}$ and note that $\eigenf_{-\mdim\z}\equiv 1$. Using that $\ln r_{\om''}^{\om'}\in\aaa\mathbb Z$ for every $\om''\in\Sigma^{m-\tilde M}$ and $\om'\in\Sigma^{\tilde M}$ and that $\sum_{\om''\in\Sigma^{m-\tilde M}}(r_{\om''}^{\om'})^{\mdim}=1$ for every fixed $\om'\in\Sigma^{\tilde M}$ since $\sum_{\om''\in\Sigma^{m-\tilde M}}(r_{\om''}^{\om'}\cdot\mathds 1_{[\om']\vp})^{\mdim}=\eigenf_{-\mdim\z}\equiv 1$, Equation (\ref{long}) simplifies to 
  \begin{eqnarray*}
    &&\sum_{i=1}^{Q-1}\sum_{\om\in\Sigma^{m}}W (r_{\om''}^{\om'})^{\mdim}\int_{\Sigma^{\infty}}\mathds 1_{\kappa\vp}(y) \ee^{-\mdim\aaa\left\lceil\frac{\psi(y)}{\aaa}+\frac{1}{\aaa}\ln{\frac{2\ee^{-T}}{d_{m}\lvert\tilde{L}_{\om'}^i\rvert}}\right\rceil}\textup{d}\nu_{-\mdim\z}(y)\\
    &&\qquad\qquad=\sum_{i=1}^{Q-1}\sum_{\om'\in\Sigma^{\tilde{M}}} W \int_{\Sigma^{\infty}}\mathds 1_{\kappa\vp}(y) \ee^{-\mdim\aaa\left\lceil\frac{\psi(y)}{\aaa}+\frac{1}{\aaa}\ln{\frac{2\ee^{-T}}{d_{m}\lvert\tilde{L}_{\om'}^i\rvert}}\right\rceil}\textup{d}\nu_{-\mdim\z}(y).
  \end{eqnarray*}
  Hence, for all $t>0$ there exists an $M'\geq \max\{M,\tilde M\}$ such that for all $T\geq M'$ we have
  \begin{eqnarray*}
    &&\hspace{-0.7cm}\ee^{-\mdim T}\leb^0(\partial F_{\ee^{-T}}\cap B(\kappa))/2\\
    &&\hspace{-0.7cm}\quad\leq (1+t)\ee^{-\mdim T}\sum_{i=1}^{Q-1}\sum_{\om'\in\Sigma^{\tilde{M}}} W \int_{\Sigma^{\infty}}\mathds 1_{\kappa\vp}(y) \ee^{-\mdim\aaa\left\lceil\frac{\psi(y)}{\aaa}+\frac{1}{\aaa}\ln{\frac{2\ee^{-T}}{d_{m}\lvert\tilde{L}_{\om'}^i\rvert}}\right\rceil}\textup{d}\nu_{-\mdim\z}(y)+c_m\ee^{-\mdim T}.
  \end{eqnarray*}
	
  Defining the function $f_{\kappa}\colon\mathbb R^+\to\mathbb R^+$ by
  \[
  f_{\kappa}(T)\defeq \ee^{-\mdim T}\sum_{i=1}^{Q-1}\sum_{\om'\in\Sigma^{\tilde{M}}} W \int_{\Sigma^{\infty}}\mathds 1_{\kappa\vp}(y) \ee^{-\mdim\aaa\left\lceil\frac{\psi(y)}{\aaa}+\frac{1}{\aaa}\ln{\frac{2\ee^{-T}}{\lvert\tilde{L}_{\om'}^i\rvert}}-\frac{1}{\aaa}\ln{A\bd_{\tilde{M}}}\right\rceil}\textup{d}\nu_{-\mdim\z}(y)
  \]
  we thus have
  \begin{eqnarray*}
    \ee^{-\mdim T}\leb^0(\partial F_{\ee^{-T}}\cap B(\kappa))/2
    \leq (1+t)(\bdneu_m\bd_m)^{\mdim}f_{\kappa}(T+\ln\bdneu_m\bd_m)+c_m\ee^{-\mdim T}.
  \end{eqnarray*}
  Likewise, 	
  \begin{eqnarray*}
    \ee^{-\mdim T}\leb^0(\partial F_{\ee^{-T}}\cap B(\kappa))/2
    \geq (1-t)(\bdneu_m\bd_m)^{-\mdim}f_{\kappa}(T-\ln\bdneu_m\bd_m).
  \end{eqnarray*}
  Clearly, $f_{\kappa}$ is periodic with period $\aaa$. Thus, Item (\ref{thislemcond2}) of \cref{periodic} is satisfied for $B=B(\kappa)$ and $f=f_{\kappa}$.
	
	In order to apply \cref{periodic} it remains to prove the validity of Item (\ref{thislemcond1}) of \cref{periodic}, that is that there exists a $\kappa\in\Delta$ for which $f_{\kappa}$ is not equal to an almost everywhere constant function.
  For that it suffices to consider the function $\tilde{f}_{\kappa}$ given by $\tilde{f}_{\kappa}(T)\defeq (A\bd_{\tilde{M}})^{-{\mdim}}f_{\kappa}(T-\ln A\bd_{\tilde{M}})$ instead. Set $\underline{\beta}\defeq \min\{\{\aaa^{-1}\ln\lvert\tilde{L}_{\om'}^i\rvert\}\mid {i=1,\ldots,Q-1}, \om'\in\Sigma^{\tilde{M}}\}$ and $\overline{\beta}\defeq\max\{\{\aaa^{-1}\ln\lvert\tilde{L}_{\om'}^i\rvert\}\mid{i=1,\ldots,Q-1}, \om'\in\Sigma^{\tilde{M}}\}$. We first assume that $\underline{\beta}>0$ and consider the following four cases.
  
  \underline{\textsc{Case 1}:} $\underline{D}\defeq\{y\in\Sigma^{\infty}\mid\{\aaa^{-1}\psi(y)\}<\underline{\beta}\}\not=\emptyset$.\\
  Since $\psi\in\mathcal{C}(\Sigma^{\infty})$ and thus $\underline{D}$ is open, there exists a $\kappa\in\Delta$ such that $\kappa\subseteq\underline{D}$. For $n\in\mathbb N$ and $r\in(0,1-\overline{\beta})$ define $T_n(r)\defeq \aaa(n+r)+\ln 2$. Then
  \[
  \tilde{f}_{\kappa}(T_n(r))=\ee^{-\mdim\aaa r}\cdot 2^{-\mdim}\sum_{i=1}^{Q-1}\sum_{\om'\in\Sigma^{\tilde{M}}} W \int_{\Sigma^{\infty}}\mathds 1_{\kappa\vp}(y) \ee^{-\mdim\aaa\left\lceil\frac{\psi(y)}{\aaa}\right\rceil+\mdim \aaa}\textup{d}\nu_{-\mdim\z}(y).
  \]
  This shows that $\tilde{f}_{\kappa}$ is strictly decreasing on $(\aaa n+\ln 2,\aaa(n+1-\overline{\beta})+\ln 2)$ for every $n\in\mathbb N$. Therefore, $f_{\kappa}$ is not equal to an almost everywhere constant function.
  
  \underline{\textsc{Case 2}:} $\overline{D}\defeq\{y\in\Sigma^{\infty}\mid\{\aaa^{-1}\psi(y)\}>\overline{\beta}\}\not=\emptyset$.\\
  Like in \textsc{Case 1}, there exists a $\kappa\in\Delta$ such that $\kappa\subseteq\overline{D}$. 
  For $n\in\mathbb N$ and $r\in(0,\underline{\beta})$ set $T_n(r)\defeq \aaa(n-r)+\ln 2$. Then
  \[
  \tilde{f}_{\kappa}(T_n(r))=\ee^{\mdim\aaa r}\cdot 2^{-\mdim}\sum_{i=1}^{Q-1}\sum_{\om'\in\Sigma^{\tilde{M}}} W \int_{\Sigma^{\infty}}\mathds 1_{\kappa\vp}(y) \ee^{-\mdim\aaa\left\lceil\frac{\psi(y)}{\aaa}\right\rceil}\textup{d}\nu_{-\mdim\z}(y).
  \]
  This shows that $\tilde{f}_{\kappa}$ is strictly decreasing on $(\aaa(n-\underline{\beta})+\ln 2,\aaa n+\ln 2)$ for every $n\in\mathbb N$. Therefore, $f_{\kappa}$ is not equal to an almost everywhere constant function.

  For the remaining cases we let $q^*\in\mathbb N\cup\{0\}$ be maximal such that $\underline{\beta}+q^*(1-\overline{\beta})\leq\overline{\beta}$.
  
  \underline{\textsc{Case 3}:}
  There exists a $q\in\{0,\ldots,q^*\}$ such that\\
  \hphantom{\underline{\textsc{Case 3}:}} 
  $D_q\defeq \{y\in\Sigma^{\infty}\mid\underline{\beta}+q(1-\overline{\beta})<\{\aaa^{-1}\psi(y)\}<\underline{\beta}+(q+1)(1-\overline{\beta})\}\neq\emptyset$.\\
  As in the above cases, there exists a $\kappa\in\Delta$ such that $\kappa\subseteq D_q$.
  For $n\in\mathbb N$ and $r\in(0,\underline{\beta})$ set $T_n^q(r)\defeq \aaa(n-\overline{\beta}+\underline{\beta}+q(1-\overline{\beta})-r)+\ln 2$. Then
  \[
  \tilde{f}_{\kappa}(T_n^q(r))=\ee^{\mdim\aaa r}\cdot 2^{-\mdim}\ee^{\mdim\aaa(\overline{\beta}-\underline{\beta}-q(1-\overline{\beta}))}\sum_{i=1}^{Q-1}\sum_{\om'\in\Sigma^{\tilde{M}}} W \int_{\Sigma^{\infty}}\mathds 1_{\kappa\vp}(y) \ee^{-\mdim\aaa\left\lceil\frac{\psi(y)}{\aaa}\right\rceil}\textup{d}\nu_{-\mdim\z}(y).
  \]
  This shows that $\tilde{f}_{\kappa}$ is strictly decreasing on $(\aaa(n-\overline{\beta}+q(1-\overline{\beta}))+\ln 2, \aaa(n-\overline{\beta}+\underline{\beta}+q(1-\overline{\beta}))+\ln 2)$.
  Therefore, $f_{\kappa}$ is not equal to an almost everywhere constant function.
  
  If neither of the cases 1-3 obtains, then the following case obtains.
  
  \underline{\textsc{Case 4}:}
  $\{y\in\Sigma^{\infty}\mid\{\aaa^{-1}\psi(y)\}\subseteq\{\underline{\beta}+q(1-\overline{\beta})\mid q\in\{0,\ldots,q^*\}\}\}=\Sigma^{\infty}$.\\
  Define $q_i\defeq\min(\{\underline{\beta}+q(1-\overline{\beta})-\{\aaa^{-1}\ln\lvert\tilde{L}_{\om'}^i\rvert\}>0\mid q\in\{0,\ldots,q^*\},\om'\in\Sigma^{\tilde{M}}\}\cup\{1\})$ and $p\defeq\min\{q_1,\ldots,q_N,1-\overline{\beta}+\underline{\beta}\}$. 
  For $n\in\mathbb N$ and $r\in(0,p/2)$ define $T_n(r)\defeq\aaa(n+r)+\ln 2$. Then
  \[
  \tilde{f}_{\Sigma^{\infty}}(T_n(r))=\ee^{-\mdim\aaa r}\cdot 2^{-\mdim}\sum_{i=1}^{Q-1}\sum_{\om'\in\Sigma^{\tilde{M}}} W \int_{\Sigma^{\infty}} \ee^{-\mdim\aaa\left\lceil\frac{\psi(y)}{\aaa}-\frac{1}{\aaa}\ln\lvert\tilde{L}_{\om'}^i\rvert\right\rceil}\textup{d}\nu_{-\mdim\z}(y).
  \] 
  This shows that $\tilde{f}_{\Sigma^{\infty}}$ is strictly decreasing on $(\aaa n+\ln 2,\aaa(n+p/2)+\ln 2)$.
  Therefore, $f_{\Sigma^{\infty}}$ is not equal to an almost everywhere constant function.
  
  If $\underline{\beta}=0$, then the same methods can be applied after shifting the origin by $(1-\overline{\beta})/2$ to the left.
  
  Thus, we can apply \cref{periodic} in all four cases and obtain that there always exists a Borel set $B(\kappa)$ such that $\underline{C}_k^f(F,B(\kappa))<\overline{C}_k^f(F,B(\kappa))$ for $k\in\{0,1\}$. 
	
  In order to deduce that the fractal curvature measures do not exist, construct a function $\eta\colon\mathbb R\to\mathbb [0,1]$ which is continuous, equal to 1 on $B(\kappa)$ and equal to 0 on $\mathbb R\setminus B(\kappa)_{\theta}$. Then $\liminf_{\e\to 0}\int \eta\e^{\mdim}\textup{d}\leb^0(\partial F_{\e}\cap\cdot)/2=\underline{C}_0^f(F,B(\kappa))<\overline{C}_0^f(F,B(\kappa))=\limsup_{\e\to 0}\int \eta\e^{\mdim}\textup{d}\leb^0(\partial F_{\e}\cap\cdot)/2$. Thus, the 0-th fractal curvature measure does not exist. Using the same function $\eta$ it follows analogously, that the 1-st fractal curvature measure does not exist, which completes the proof.
\end{proof}

\begin{proof}[Proof of Part {\rm(\ref{curvatureresult:lattice})} of \cref{curvatureresult} (boundedness and positivity)]
Since $\xi$ is a lattice function, there exist $\z,\psi\in\mathcal C(\Sigma^{\infty})$ such that
\[
	\xi-\z=\psi-\psi\circ\sigma
\]
and such that $\z$ is a function whose range is contained in a discrete subgroup of $\mathbb R$. Let $\aaa>0$ be the maximal real number such that $\z(\Sigma^{\infty})\subseteq\aaa\mathbb Z$. Recall from the beginning of Section \ref{sec:proofs} that the hypotheses and \cref{rmk:alpha} imply that $\xi$ is Hölder continuous and strictly positive and that the unique $s>0$ for which $\eigenv_{-s\xi}=1$ is the Minkowski dimension $\mdim$ of $F$.

Fix the notation of the beginning of Section \ref{sec:proofs}. Since $\mathds 1_{Z\vp}$ is Hölder continuous and since we can assume that $\mathds 1_{Z\vp}$ is not identically zero, by combining equations (\ref{eq:Aomover}), (\ref{eq:Aomunder}) and (\ref{eq:Snpm}), we see that an application of \cref{thmlalleyneu} to $\overline{A}_{\om}^i(x,\e,Z)$ and $\underline{A}_{\om}^i(x,\e,Z)$ gives the following asymptotics.
\begin{eqnarray}
  \hspace{-1cm}\overline{A}_{\om}^i(x,\e,Z)
  \hspace{-0.2cm}&\sim&\hspace{-0.2cm} W_{\om}(x)\int_{\Sigma^{\infty}}\mathds 1_{Z\vp}(y) \ee^{-\mdim\aaa\left\lceil\frac{\psi(y)-\psi(\om x)}{\aaa}+\frac{1}{\aaa}\ln\frac{2\e}{\Lomi\bd_{n(\om)}}\right\rceil}\textup{d}\nu_{-\mdim\z}(y)\ \text{and}\label{eq:Aomoverasym:lattice}\\
  \hspace{-1cm}\underline{A}_{\om}^i(x,\e,Z)
  \hspace{-0.2cm}&\sim&\hspace{-0.2cm} W_{\om}(x)\int_{\Sigma^{\infty}}\mathds 1_{Z\vp}(y) \ee^{-\mdim\aaa\left\lceil\frac{\psi(y)-\psi(\om x)}{\aaa}+\frac{1}{\aaa}\ln\frac{2\e\bd_{n(\om)}}{\Lomi}\right\rceil}\textup{d}\nu_{-\mdim\z}(y)\label{eq:Aomunderasym:lattice}
\end{eqnarray}
as $\e\to 0$ uniformly for $x\in\Sigma^{\infty}$, where
\begin{eqnarray}\label{eq:W}
  W_{\om}(x)\defeq \frac{\aaa\eigenf_{-\mdim\z}(\om x)}{(1-\ee^{-\mdim\aaa})\int\z\textup{d}\mu_{-\mdim\z}}.
\end{eqnarray}

  For the boundedness we first remark that $\overline{C}_0^f(F,\cdot)$ is monotonically increasing as a set function in the second component. Therefore, in order to find an upper bound for $\overline{C}_0^f(F,\cdot)$ it suffices to consider $\overline{C}_0^f(F,\mathbb R)$. For all $m\in\mathbb N$ we have 
  \begin{eqnarray*}
    \hspace{-1cm}&&\overline{C}_0^f(F,\mathbb R)
    \stackrel{(\ref{eq:fcmoben})}{\leq}\limsup_{\e\to 0}\e^{\mdim}\sum_{i=1}^{Q-1}\sum_{\om\in\Sigma^m}\overline{A}_{\om}^i(x,\e,\Sigma^{\infty})\\
    \hspace{-1cm}&&\qquad\stackrel{(\ref{eq:Aomoverasym:lattice})}{=} 	\limsup_{\e\to 0}\e^{\mdim}\sum_{i=1}^{Q-1}\sum_{\om\in\Sigma^m}
    W_{\om}(x)\int_{\Sigma^{\infty}} \ee^{-\mdim\aaa\left\lceil\frac{\psi(y)-\psi(\om x)}{\aaa}+\frac{1}{\aaa}\ln\frac{2\e}{\Lomi\bd_{m}}\right\rceil}\textup{d}\nu_{-\mdim\z}(y)\\
    \hspace{-1cm}&&\qquad\ \ \leq\ \ \limsup_{\e\to 0}\e^{\mdim}\sum_{i=1}^{Q-1}\sum_{\om\in\Sigma^m}
    W_{\om}(x) \int_{\Sigma^{\infty}} \ee^{-\mdim\aaa\left(\frac{\psi(y)-\psi(\om x)}{\aaa}+\frac{1}{\aaa}\ln\frac{2\e}{\Lomi\bd_{m}}\right)}\textup{d}\nu_{-\mdim\z}(y)\\
    \hspace{-1cm}&&\qquad\ \ =\ \ \sum_{i=1}^{Q-1}\sum_{\om\in\Sigma^m} \frac{\aaa \ee^{\mdim\psi(\om x)}\eigenf_{-\mdim\z}(\om x)}{(1-\ee^{-\mdim\aaa})\int\z\textup{d}\mu_{-\mdim\z}} \left(\frac{\Lomi\bd_m}{2}\right)^{\mdim}\int_{\Sigma^{\infty}} \ee^{-\mdim\psi(y)}\textup{d}\nu_{-\mdim\z}(y).
  \end{eqnarray*}
  Note that $\eigenf_{-\mdim\xi}=\ee^{\mdim\psi}\eigenf_{-\mdim\z}$ and $\textup{d}\nu_{-\mdim\xi}=\ee^{-\mdim\psi}\textup{d}\nu_{-\mdim\z}$. Hence, by \cref{Upsilon}
  \begin{eqnarray*}
    \overline{C}_0^f(F,\mathbb R)
    \leq \liminf_{m\to\infty}\sum_{i=1}^{Q-1}\sum_{\om\in\Sigma^m}\Lomi^{\mdim}\frac{\aaa 2^{-\mdim}}{(1-\ee^{-\mdim\aaa})\int\z\textup{d}\mu_{-\mdim\z}}\eqdef c_0.
  \end{eqnarray*}
  $c_0\in(0,\infty)$ because $\sum_{i=1}^{\Q-1}\sum_{\om\in\Sigma^m}\Lomi^{\mdim}\leq\sum_{i=1}^{\Q-1}\sum_{\om\in\Sigma^m}\|\phi'_{\om}\|^{\mdim}\eqdef a_m$, where $\|\cdot\|$ denotes the supremum-norm on $\mathcal C(X)$ and the sequence $(a_m)_{m\in\mathbb N}$ is bounded by Lemma 4.2.12 of \cite{Urbanski_Buch}.
  
  That $\underline{C}_0^f(F,\mathbb R)$ is positive can be seen by the following.
  \begin{eqnarray*}
    \hspace{-0.8cm}&&\underline{C}_0^f(F,\mathbb R)
    \geq\liminf_{\e\to 0}\e^{\mdim}\sum_{i=1}^{Q-1}\sum_{\om\in\Sigma^m}\underline{A}_{\om}^i(x,\e,\Sigma^{\infty})\\
    \hspace{-0.7cm}&&\qquad\stackrel{(\ref{eq:Aomunderasym:lattice})}{\geq}\liminf_{\e\to 0}\e^{\mdim}\sum_{i=1}^{Q-1}\sum_{\om\in\Sigma^m} W_{\om}(x) \int_{\Sigma^{\infty}} \ee^{-\mdim\aaa\left(\frac{\psi(y)-\psi(\om x)}{\aaa}+\frac{1}{\aaa}\ln\frac{2\e\bd_{m}}{\Lomi}+1\right)}\textup{d}\nu_{-\mdim\z}(y)\\
    \hspace{-0.9cm}&&\qquad\ \ =\ \  \sum_{i=1}^{Q-1}\sum_{\om\in\Sigma^m} \frac{\aaa\eigenf_{-\mdim\z}(\om x)}{(1-\ee^{-\mdim\aaa})\int\z\textup{d}\mu_{-\mdim\z}} \ee^{\mdim\psi(\om x)-\mdim\aaa}\left(\frac{\Lomi}{2\bd_m}\right)^{\mdim}\int_{\Sigma^{\infty}} \ee^{-\mdim\psi(y)}\textup{d}\nu_{-\mdim\z}(y).
  \end{eqnarray*}
  By using $\eigenf_{-\mdim\xi}=\ee^{\mdim\psi}\eigenf_{-\mdim\z}$ and $\textup{d}\nu_{-\mdim\xi}=\ee^{-\mdim\psi}\textup{d}\nu_{-\mdim\z}$ and \cref{Upsilon}, we hence obtain
  \[
  \underline{C}_0^f(F,\mathbb R)
  \geq \limsup_{m\to\infty}\sum_{i=1}^{Q-1}\sum_{\om\in\Sigma^m}\Lomi^{\mdim}\frac{\aaa 2^{-\mdim}\ee^{-\mdim\aaa}}{(1-\ee^{-\mdim\aaa})\int\z\textup{d}\mu_{-\mdim\z}}>0.
  \]
  
  The results on $\underline{C}_1^f(F,B)$ and $\overline{C}_1^f(F,B)$ are now a straightforward application of \cref{lemWinterRataj}.
\end{proof}

\begin{proof}[Proof of \cref{conformalMinkowski}]
  Parts (\ref{conformalMinkowski:average}) and (\ref{conformalMinkowski:nonlattice}) are immediate consequences of \cref{curvatureresult}. For Part (\ref{conformalMinkowski:lattice}) use that  the hypotheses of Part (\ref{conformalMinkowski:lattice}) of \cref{conformalMinkowski} and \cref{lem:existanceB} together imply that  
  \begin{eqnarray*}
    \overline{A}
    &\defeq&\lim_{\e\to 0}\e^{\mdim}\int_{\Sigma^{\infty}}\ee^{-\mdim\aaa\left\lceil\frac{\psi(y)-\psi(\om x)}{\aaa}+\frac{1}{\aaa}\ln\frac{2\e}{\Lomi\bd_m}\right\rceil}\textup{d}\nu_{-\mdim\z}(y)\cdot \left(\frac{2}{\Lomi\bd_m}\right)^{\mdim}\quad\text{and}\\
    \underline{A}
    &\defeq&\lim_{\e\to 0}\e^{\mdim}\int_{\Sigma^{\infty}}\ee^{-\mdim\aaa\left\lceil\frac{\psi(y)-\psi(\om x)}{\aaa}+\frac{1}{\aaa}\ln\frac{2\e\bd_m}{\Lomi}\right\rceil}\textup{d}\nu_{-\mdim\z}(y)\cdot \left(\frac{2\bd_m}{\Lomi}\right)^{\mdim}
  \end{eqnarray*}
  exist for every $\om\in\Sigma^m$ and $i\in\{1,\ldots,Q-1\}$, are independent of $\om$ and $i$ and are equal, that is $\overline{A}=\underline{A}\eqdef A$.
	Combining Equations (\ref{eq:fcmoben}) and (\ref{eq:Aomoverasym:lattice}) and Equations (\ref{eq:fcmunten}) and (\ref{eq:Aomunderasym:lattice}), where $Z=\Sigma^{\infty}$, we conclude
  \begin{eqnarray*}
    \overline{C}_0^f(F,\mathbb R)
    &\leq&\sum_{i=1}^{Q-1}\sum_{\om\in\Sigma^m} W_{\om}(x)\left(\frac{\Lomi\bd_m}{2}\right)^{\mdim}\cdot A\quad\text{and}\\
    \underline{C}_0^f(F,\mathbb R)
    &\geq&\sum_{i=1}^{Q-1}\sum_{\om\in\Sigma^m} W_{\om}(x)\left(\frac{\Lomi}{2\bd_m}\right)^{\mdim}\cdot A, 		
  \end{eqnarray*}
  where $W_{\om}(x)$ is as defined in (\ref{eq:W}).
  Applying \cref{Upsilon} we obtain $\overline{C}_0^f(F,\mathbb R)=\underline{C}_0^f(F,\mathbb R)$. An application of \cref{lemWinterRataj} then completes the proof.
\end{proof}

\subsection{Average Fractal Curvature Measures}
\begin{proof}[Proof of Part {\rm(\ref{curvatureresult:average})} of \cref{curvatureresult}]
  If $\xi$ is nonlattice, Part (\ref{curvatureresult:average}) of \cref{curvatureresult} immediately follows from Part (\ref{curvatureresult:nonlattice}) of \cref{curvatureresult} and the fact that $f(\e)\sim c$ as $\e\to 0$ for some constant $c\in\mathbb R$ implies $\lim_{T\searrow 0}\lvert\ln T\rvert^{-1}\int_T^1 \e^{-1}f(\e)\textup{d}\e=c$ for every locally integrable function $f\colon(0,\infty)\to\mathbb R$. 
  
  Thus for the rest of the proof we assume that $\xi$ is lattice and fix the notation from the beginning of Section \ref{sec:proofs}. In particular, recall that $b\in\mathbb R\setminus F$.
	
  We begin with showing the result on the 0-th average fractal curvature measure.
  
  Observe that $\lim_{T\searrow 0}\lvert\ln T\rvert^{-1}\int_T^1c\e^{\mdim-1}\textup{d}\e=\lim_{T\to \infty}\lvert T\rvert^{-1}\int_0^T c\ee^{-\mdim t}\textup{d}t=0$ for every constant $c\in\mathbb R$.
  For a fixed $m\in\mathbb N$ define $M\defeq \min\{\Lomi\mid i\in\{1,\ldots,Q-1\}, \om\in\Sigma^m\}/2$. From Equations (\ref{eq:Xisum}) and (\ref{eq:Aomover}) we deduce the following. 
  \begin{eqnarray*}
    &&\hspace{-0.8cm}\overline{D}
    \defeq\limsup_{T\searrow 0}\lvert2\ln T\rvert^{-1}\int_T^1 \e^{\mdim-1}\leb^0(\partial F_{\e}\cap(-\infty,\bb])\textup{d}\e\\
    &&\hspace{-0.8cm}\leq \limsup_{T\searrow 0}\lvert\ln T\rvert^{-1}\hspace{-0.05cm}\bigg(\hspace{-0.1cm}\int_T^M\hspace{-0.25cm} \e^{\mdim-1}\sum_{i=1}^{Q-1}\hspace{-0.1cm}\sum_{\om\in\Sigma^m} \overline{A}_{\om}^i(x,\e,Z)\textup{d}\e\hspace{-0.05cm} + \hspace{-0.05cm}\frac{1}{2}\int_M^1\hspace{-0.2cm} \e^{\mdim-1}\leb^0(\partial F_{\e}\cap(-\infty,\bb])\textup{d}\e\hspace{-0.1cm}\bigg).
  \end{eqnarray*}
  Local integrability of the integrands implies that we have the following equation for all $m\in\mathbb N$.
  \begin{eqnarray}
    \hspace{-0.6cm}\overline{D}
    \hspace{-0.3cm}&\leq&\hspace{-0.3cm} \limsup_{T\searrow 0}\lvert\ln T\rvert^{-1}\int_T^1 \e^{\mdim-1}\sum_{i=1}^{Q-1}\sum_{\om\in\Sigma^m} \overline{A}_{\om}^i(x,\e,Z)\textup{d}\e\nonumber\\
    \hspace{-0.3cm}&=&\hspace{-0.3cm} \limsup_{T\to \infty} T^{-1}\sum_{i=1}^{Q-1}\sum_{\om\in\Sigma^m}\int_0^T \ee^{-\mdim t} \overline{A}_{\om}^i(x,\ee^{-t},Z)\textup{d}t\nonumber\\
    \hspace{-0.3cm}&\stackrel{(\ref{eq:Snpm})}{=}&\hspace{-0.3cm} \limsup_{T\to\infty} T^{-1}\sum_{i=1}^{Q-1}\sum_{\om\in\Sigma^m} \int_0^T \ee^{-\mdim t}\sum_{n=0}^{\infty}	\sum_{y\colon\sigma^n y=\om x}\mathds 1_{Z\vp}(y)\cdot\mathds 1_{\{S_n\xi(y)\leq t-\ln\frac{2}{\Lomi\bd_{m}}\}\vp}\textup{d}t\nonumber\\
    \hspace{-0.3cm}&\leq&\hspace{-0.3cm} \limsup_{T\to\infty} \sum_{i=1}^{Q-1}\sum_{\om\in\Sigma^m}\left(\frac{\Lomi\bd_m}{2}\right)^{\mdim}\frac{T-\ln\frac{2}{\Lomi\bd_{m}}}{T}\cdot\nonumber\\
    \hspace{-0.3cm}&&\hspace{-0.3cm}\quad \left(T-\ln\frac{2}{\Lomi\bd_{m}}\right)^{-1}\hspace{-0.2cm}\int_0^{T-\ln\frac{2}{\Lomi\bd_{m}}} \ee^{-\mdim t}\sum_{n=0}^{\infty} \sum_{y\colon\sigma^n y=\om x}\mathds 1_{Z\vp}(y)\cdot\mathds 1_{\{S_n\xi(y)\leq t\}\vp}\textup{d}t\nonumber\\
    \hspace{-0.3cm}&=&\hspace{-0.3cm} \sum_{i=1}^{Q-1}\sum_{\om\in\Sigma^m}\left(\frac{\Lomi\bd_m}{2}\right)^{\mdim} \frac{\eigenf_{-\mdim\xi}(\om x)\int\mathds 1_{Z\vp}\textup{d}\nu_{-\mdim\xi}}{\mdim\int\xi\textup{d}\mu_{-\mdim\xi}}.\label{eq:avcor}	
  \end{eqnarray}
  The last equality is an application of \cref{thmlalleyaverage}. Because (\ref{eq:avcor}) holds for all $m\in\mathbb N$, applying \cref{Upsilon} yields
  \begin{eqnarray}
    &&\limsup_{T\searrow 0}\lvert2\ln T\rvert^{-1}\int_T^1 \e^{\mdim-1}\leb^0(\partial F_{\e}\cap(-\infty,\bb])\textup{d}\e\nonumber\\
      &&\qquad\qquad
      \leq \frac{2^{-\mdim}\int\mathds 1_{Z\vp}\textup{d}\nu_{-\mdim\xi}}{\mdim\int\xi\textup{d}\mu_{-\mdim\xi}} \liminf_{m\to\infty}\sum_{i=1}^{Q-1}\sum_{\om\in\Sigma^m} \Lomi^{\mdim}.\label{eq:averageoben}
  \end{eqnarray}
  Analogous estimates give 
  \begin{eqnarray}
    &&\liminf_{T\searrow 0}\lvert2\ln T\rvert^{-1}\int_T^1 \e^{\mdim-1}\leb^0(\partial F_{\e}\cap(-\infty,\bb])\textup{d}\e\nonumber\\
      &&\qquad\qquad
      \geq \frac{2^{-\mdim}\int\mathds 1_{Z\vp}\textup{d}\nu_{-\mdim\xi}}{\mdim\int\xi\textup{d}\mu_{-\mdim\xi}} \limsup_{m\to\infty}\sum_{i=1}^{Q-1}\sum_{\om\in\Sigma^m} \Lomi^{\mdim}.\label{eq:averageunten}
  \end{eqnarray}
  Equations (\ref{eq:averageoben}) and (\ref{eq:averageunten}) together imply that for every $b\in\mathbb R\setminus F$
  \begin{eqnarray*}
    \lim_{T\searrow 0}\lvert2\ln T\rvert^{-1}\int_T^1 \e^{\mdim-1}\leb^0(\partial F_{\e}\cap(-\infty,\bb])\textup{d}\e
      = \frac{2^{-\mdim}c}{\entro_{\mu_{-\mdim\xi}}}\nu(F\cap(-\infty,\bb]),
  \end{eqnarray*}
  where the constant $c\defeq\lim_{m\to\infty}\sum_{i=1}^{Q-1}\sum_{\om\in\Sigma^m}\Lomi^{\mdim}$ is well-defined.
  Since $\mathbb R\setminus F$ is dense in $\mathbb R$, the statement on the 0-th average fractal curvature measure in Part (\ref{curvatureresult:average}) of \cref{curvatureresult} follows.
  
  For the statement on the 1-st average fractal curvature measure, we use Part (\ref{curvatureresult:lattice}) of \cref{curvatureresult} which says that $\overline{C}_0^f(F,(-\infty,\bb])<\infty$ for every $\bbb\in\mathbb R\setminus F$. Applying \cref{lemWinterRataj} hence yields that $\overline{\mathcal M}(F\cap(-\infty,\bb])<\infty$ for every $\bbb\in\mathbb R\setminus F$. By the same arguments that were used in the end of the proof of Part (\ref{curvatureresult:nonlattice}), we can thus apply \Cref{lemWinterRataj:average} to $F\cap(-\infty,\bbb]$ and obtain the desired statement.
\end{proof}

\section{Proofs concerning the Special Cases}\label{sec:specialcases}

\subsection{Self-Similar Sets; Proof of \cref{similars}}
Self-similar sets satisfying the open set condition form a special class of self-conformal sets, namely those which are generated by an iterated function system $\Phi$ consisting of similarities $\phi_1,\ldots,\phi_N$.
We let $r_1,\ldots,r_N$ denote the respective similarity ratios of $\phi_1,\ldots,\phi_N$ and set $r_{\om}\defeq r_{\om_1}\cdots r_{\om_n}$ for a finite word $\om=\om_1\cdots\om_n\in\Sigma^n$.
When considering self-similar sets some of the formulae simplify significantly:
\renewcommand{\theenumi}{\Alph{enumi}}
\begin{enumerate}
\item\label{similarpotential} The geometric potential function is constant on the one-cylinders meaning $\xi(\om)=-\ln r_{\om_1}$ for $\om=\om_1\om_2\cdots\in\Sigma^{\infty}$.
\item The unique $\sigma$-invariant Gibbs measure $\mu_{-\mdim\xi}$ for the potential function $-\mdim\xi$ coincides with the $\mdim$-dimensional normalised Hausdorff measure on $F$. Thus, $\mu_{-\mdim\xi}([i])=r_i^{\mdim}$, where $[i]$ shall denote the cylinder of $i\in\Sigma$. Therefore we have that $\entro_{\mu_{-\mdim\xi}}=-\mdim\sum_{i\in\Sigma}\ln(r_i)r_i^{\mdim}$. 
\item The lengths of the \main\ gaps of $\phi_{\om}F$ are just multiples of the lengths of the primary gaps of $F$, that is $\Lomi=r_{\om}\lvert L^{i}\rvert$ for each $i\in\{1,\ldots,\Q-1\}$ and $\om\in\Sigma^*$.
\item\label{MoranHutchinson} By the Moran-Hutchinson formula (see for instance Theorem 9.3 of \cite{Falconer_Foundation}) we have that $\sum_{\om\in\Sigma^n}r_{\om}^{\mdim}=1$ for each $n\in\mathbb N$.
\end{enumerate}

\begin{proof}[Proof of \cref{similars}]
	Combining (\ref{similarpotential})-(\ref{MoranHutchinson}) with \cref{curvatureresult}, we obtain Part (\ref{ss:average}) of \cref{similars}.
In order to prove Part (\ref{ss:lattice}) of \cref{similars}, which actually is a stronger result than that of Part (\ref{curvatureresult:lattice}) of \cref{curvatureresult}, we are going to make use of the asymptotics (\ref{eq:Aomoverasym:lattice}) and (\ref{eq:Aomunderasym:lattice}) that we obtained for self-conformal sets.

  As $F\cap B$ has got a representation as a finite nonempty union of sets of the form $\phi_{\om}F$ with $\om\in\Sigma^*\setminus\{\emptyset\}$, there is a set $Z\subseteq \Sigma^{\infty}$ which is a finite union of cylinder sets and which satisfies $\bij Z=F\cap B$. For this $Z$, $\mathds 1_{Z\vp}$ is Hölder continuous.
  Furthermore, the range of the geometric potential function of a lattice self-similar set itself is contained in a discrete subgroup of $\mathbb R$. Thus, $\psi$ is a constant function and $\z=\xi$. Moreover, $\bd_m=1$ for al $m\in\mathbb N$ and one easily verifies that $\eigenf_{-\mdim\xi}\equiv 1$ and $\Lomi=r_{\om}\lvert L^i\rvert$. For these reasons the methods in the beginning of Section \ref{sec:proofs} simplify in the following way. 
  
  Let $T\geq 0$ be sufficiently large such that $F_{\ee^{-T}}\cap B=(F\cap B)_{\ee^{-T}}$ and let $x\in\Sigma^{\infty}$ be arbitrary. Then there exists a constant $c\geq 0$, which depends on the number of sets $\phi_{\om}F$ whose union is $F\cap B$, such that
  \begin{eqnarray}
    &&\leb^0\big(\partial F_{\ee^{-T}}\cap B \big)/2
    \stackrel{(\ref{eq:lebesgue})}{=}\sum_{i=1}^{Q-1}\card\{\om\in\Sigma^*\mid L_{\om}^i\subseteq B,\ \Lomi\geq 2\ee^{-T}\}+c\nonumber\\
    &&\qquad =\sum_{i=1}^{Q-1}\sum_{n=0}^{\infty}\sum_{\om\in\Sigma^n}\mathds 1_{Z\vp}(\om x)\mathds 1_{\{\lvert\phi'_{\om}(x)\rvert\cdot\lvert L^i\rvert\geq 2\ee^{-T}\}\vp}+c\nonumber\\
    &&\qquad =\sum_{i=1}^{Q-1}\sum_{n=0}^{\infty}\sum_{y\colon\sigma^n y=x} \mathds 1_{Z\vp} (y)\mathds 1_{\{S_n\xi(y)\leq-\ln\frac{2\ee^{-T}}{\lvert L^i\rvert}\}\vp}+c\nonumber\\
    &&\qquad\sim \sum_{i=1}^{Q-1}\frac{\aaa\int\mathds 1_{Z\vp}\textup{d}\nu_{-\mdim\xi}}{(1-\ee^{-\mdim\aaa})\int\xi\textup{d}\mu_{-\mdim\xi}}\cdot \ee^{-\mdim\aaa\left\lceil\ln\frac{2\ee^{-T}}{\lvert L^i\rvert}\right\rceil}+c\qquad\text{(as $T\to\infty$)},\label{ssasymptoticproof}
  \end{eqnarray}
  where the last asymptotic is obtained by applying \cref{thmlalleyneu}.
  We introduce the function $f\colon\mathbb R^+\to\mathbb R^+$ given by
  \[
  f(T)\defeq \ee^{-\mdim T}\frac{\aaa \nu(B)}{(1-\ee^{-\mdim\aaa})\entro_{\mu_{-\mdim\xi}}}\sum_{i=1}^{Q-1} \ee^{-\mdim\aaa\left\lceil\frac{1}{\aaa}\ln\frac{2\ee^{-T}}{\lvert L^i\rvert}\right\rceil}.
  \]
  By the asymptotics given in (\ref{ssasymptoticproof}), we know that for all $t>0$ there exists an $M\in\mathbb N$ such that for all $T\geq M$ we have
  \begin{eqnarray*}
    (1-t)f(T)
    \leq \ee^{-\mdim T}\leb^0(\partial F_{\ee^{-T}}\cap B)/2
    \leq(1+t)f(T)+c\ee^{-\mdim T}.
  \end{eqnarray*}
  Clearly, $f$ is a periodic function with period $\aaa$. Moreover, $f$ is piecewise continuous with a finite number of discontinuities in an interval of length $\aaa$. Additionally, on every interval where $f$ is continuous, $f$ is strictly decreasing. Therefore $f$ is not equal to an almost everywhere constant function. Thus, all conditions of \cref{periodic} are satisfied which finishes the proof.
\end{proof}

\subsection{\boldmath{$\mathcal C^{1+\alpha}$} Images of Self-Similar Sets; Proofs of \cref{corimage,corminim}}
In this subsection we consider the case that $F$ is an image of a self-similar set $K\subseteq X$ under a conformal map $g\in\mathcal C^{1+\alpha}(\mathcal U)$, where $\alpha> 0$ and $\mathcal U$ is a convex neighbourhood of the compact connected set $X$. We assume that $\lvert g'\rvert$ is bounded away from 0 on its domain of definition. Thus, $g$ is bi-Lipschitz and therefore the Minkowski dimension of $F$ coincides with the Minkowski dimension of $K$ (see for instance Corollary 2.4 of \cite{Falconer_Foundation}). We denote the common value by $\mdim$.

The similarities $R_1,\ldots, R_N$ generating $K$ and the mappings $\phi_1,\ldots,\phi_N$ generating $F$ are connected through the equations $\phi_i=g\circ R_i\circ g^{-1}$ for each $i\in\Sigma$. We denote by $\bij_K$ and $\bij_F$ respectively the natural code maps from $\Sigma^{\infty}$ to $K$ and $F$. 
If we further let $\mathcal H_K^{\mdim}$ denote the normalised $\mdim$-dimensional Hausdorff measure on $K$, that is $\mathcal H_K^{\mdim}(\cdot)\defeq\mathcal H^{\mdim}(\cdot\cap K)/\mathcal H^{\mdim}(K)$, and let $r_1,\ldots,r_N$ denote the respective similarity ratios of $R_1,\ldots, R_N$, we have the following list of observations.

\renewcommand{\theenumi}{\Alph{enumi}'}
\begin{enumerate}
\item\label{en:diffbar} $\phi_i$ is differentiable for every $i\in\Sigma$ with differential
  \[
  \phi'_i(y)=\frac{g'(R_i\circ g^{-1}(y))}{g'(g^{-1}(y))}\cdot r_i,
  \] 
  where $y\in Y$ and $Y$ is the nonempty compact interval which each $\phi_i$ is defined on.
\item\label{en:geopot} The geometric potential function $\xi_F$ associated with $F$ is given by $\xi_F(\om)=-\ln\lvert g'(g^{-1}(\bij_F\om))\rvert +\ln\lvert g'(g^{-1}(\bij_F\sigma\om))\rvert-\ln r_{\om_1}$, where $\om=\om_1\om_2\cdots\in\Sigma^{\infty}$.
The geometric potential function $\xi_K$ associated with $K$ is given by $\xi_K(\om)=-\ln r_{\om_1}$.  Thus $\xi_K$ is nonlattice, if and only if $\xi_F$ is nonlattice.
\item\label{en:gibbs} The unique $\sigma$-invariant Gibbs measure for the potential function $-\mdim\xi_F$ is $\mu_{-\mdim\xi_F}=\mathcal H_K^{\mdim}\circ g^{-1}\circ\bij_F$, the one associated with $-\mdim\xi_K$ is $\mu_{-\mdim\xi_K}=\mathcal H_K^{\mdim}\circ\bij_F$. 
\item From (\ref{en:geopot}) and (\ref{en:gibbs}) we obtain 
  \[
  \hspace{1.5cm}\entro_{-\mdim\xi_F}
  = \int_{\Sigma^{\infty}}\xi_F\textup{d}\mu_{-\mdim\xi_F}
  = -\sum_{i\in\Sigma}\ln r_i \cdot r_i^{\mdim}
  = \int_{\Sigma^{\infty}}\xi_K\textup{d}\mu_{-\mdim\xi_K}
  =\entro_{-\mdim\xi_K}.
  \]
\end{enumerate} 

Further, let $\tilde{L}^1,\ldots,\tilde{L}^{\Q-1}$ denote the primary gaps of $K$ and $\tilde{L}_{\om}^1,\ldots,\tilde{L}_{\om}^{\Q-1}$ the \main\ gaps of $R_{\om}K$ for each $\om\in\Sigma^*$ and recall that $L^1,\ldots,L^{\Q-1}$ and $L_{\om}^1,\ldots,L_{\om}^{\Q-1}$ respectively denote the primary gaps of $F$ and the \main\ gaps of $\phi_{\om}F$. Then
\begin{enumerate}\setcounter{enumi}{4}
\item $L_{\om}^i=g(\tilde{L}_{\om}^{i})$ for $i\in\{1,\ldots,\Q-1\}$ and $\om\in\Sigma^*$. Since furthermore $\lvert\tilde{L}_{\om}^{i}\rvert=r_{\om}\lvert\tilde{L}^{i}\rvert$, we have 
  \begin{eqnarray*}
  \hspace{1cm}\lim_{n\to\infty}\sum_{i=1}^{\Q-1}\sum_{\om\in\Sigma^n}\Lomi^{\mdim}
  &=&\lim_{n\to\infty}\sum_{i=1}^{Q-1}\sum_{\om\in\Sigma^n}\left(r_{\om}\lvert\tilde{L}^i\rvert\cdot\lvert g'(x_{\om})\rvert\right)^{\mdim}\\
  &=&\sum_{i=1}^{\Q-1}\lvert \tilde{L}^{i}\rvert^{\mdim}\int_K \lvert g'\rvert^{\mdim}\textup{d}\mathcal H_K^{\mdim},
  \end{eqnarray*}
  where $x_{\om}\in[\om]$ for each $\om\in\Sigma^*$. Note that the above line can be rigorously proven by using the Bounded Distortion Lemma (\cref{bd}).
\item\label{en:nu} The $\mdim$-conformal measure $\nu_F$ associated with $F$ and the $\mdim$-conformal measure $\nu_K$ associated with $K$ are absolutely continuous with Radon-Nikodym derivative
  \[
  \frac{\textup{d}\nu_F}{\textup{d}\nu_K\circ g^{-1}}=\lvert g'\circ g^{-1}\rvert^{\mdim}\bigg(\int_K \lvert g'\rvert^{\mdim}\textup{d}\mathcal H_K^{\mdim}\bigg)^{-1}.
  \]
\item\label{en:conformal} From the fact that $R_1,\ldots, R_N$ are contractions and $g'$ is Hölder continuous and bounded away from 0, one can deduce that there exists an iterate of $\Phi\defeq\{\phi_1,\ldots,\phi_N\}$ which solemnly consists of contractions. As this iterate also generates $F$, it follows that $F$ is a self-conformal set.
\end{enumerate}

\begin{proof}[Proof of \cref{corimage}]
Using (\ref{en:diffbar})-(\ref{en:conformal}) an application of Parts (\ref{curvatureresult:average}) and (\ref{curvatureresult:nonlattice}) of \cref{curvatureresult} to $F$ and of \cref{similars} to $K$ proves \cref{corimage}.
\end{proof}

\begin{proof}[Proof of \cref{corminim}]
	Parts (\ref{minim:average}) and (\ref{minim:nonlattice}) of \cref{corminim} are immediate consequences of \cref{corimage}. 
	Part (\ref{minim:lattice}) of \cref{corminim} is going to be deduced from Part (\ref{conformalMinkowski:lattice}) of \cref{conformalMinkowski}. We let $\bij_K$ and $\bij_F$ respectively denote the natural code maps from $\Sigma^{\infty}$ to $K$ and $F$ and observe that $\bij_K=g_n^{-1}\circ\bij_F$.
  Further, we let $\xi_{n}$ denote the geometric potential function associated with $F_n$. By Property (\ref{en:geopot}) we see that $\xi_{n}-\xi_K=\psi-\psi\circ\sigma$, where $\psi\defeq-\ln g'_n\circ \bij_K$.
  By definition we have that $g'_n(x)=\left(\tilde{g}(x)(\ee^{\mdim\aaa n}-1)+1\right)^{-1/\mdim}$ for $x\in[-1,\infty)$. Thus, $\psi(\Sigma^{\infty})=-\ln g'_n\circ\bij_K(\Sigma^{\infty})\subseteq[0,\aaa n]$. We now show that Condition (\ref{existencecondition}) from \cref{conformalMinkowski} is satisfied. 
    \begin{eqnarray*}
      &&\hspace{-0.5cm}\sum_{i=0}^n\ee^{-\mdim\aaa i}\nu_{-\mdim\z}\circ\psi^{-1}([\aaa i,\aaa i+t))
        =\sum_{i=0}^n\ee^{-\mdim\aaa i}\nu_K\circ \tilde{g}^{-1}\left(\left[\frac{\ee^{\mdim\aaa i}-1}{\ee^{\mdim\aaa n}-1},\frac{\ee^{\mdim\aaa i+\mdim t}-1}{\ee^{\mdim\aaa n}-1}\right)\right)\\
          &&\hspace{-0.5cm}\qquad =\sum_{i=0}^{n-1}\frac{\ee^{\mdim t}-1}{\ee^{\mdim\aaa n}-1}
          =\frac{\ee^{\mdim t}-1}{\ee^{\mdim\aaa}-1}\sum_{i=0}^{n}\ee^{-\mdim\aaa i}\nu_K\circ \tilde{g}^{-1}\left(\left[\frac{\ee^{\mdim\aaa i}-1}{\ee^{\mdim\aaa n}-1},\frac{\ee^{\mdim\aaa(i+1)}-1}{\ee^{\mdim\aaa n}-1}\right)\right)\\
            &&\hspace{-0.5cm}\qquad=\frac{\ee^{\mdim t}-1}{\ee^{\mdim\aaa}-1}\sum_{i=0}^{n}\ee^{-\mdim\aaa i}\nu_{-\mdim\z}\circ\psi^{-1}([\aaa i,\aaa(i+1)))
    \end{eqnarray*}
    holds for all $t\in[0,\aaa)$ which completes the proof.
\end{proof}
\addcontentsline{toc}{section}{References}
\providecommand{\bysame}{\leavevmode\hbox to3em{\hrulefill}\thinspace}
\providecommand{\MR}{\relax\ifhmode\unskip\space\fi MR }
\providecommand{\MRhref}[2]{%
  \href{http://www.ams.org/mathscinet-getitem?mr=#1}{#2}
}
\providecommand{\href}[2]{#2}

\end{document}